\documentclass[11pt,letterpapert]{article}
\usepackage{graphicx}
\usepackage{amssymb}
\usepackage{epstopdf,mathtools,enumitem}
\usepackage{fancyhdr}
\usepackage{amsmath}
\usepackage{amsthm}
\usepackage{hyperref}
\usepackage{makeidx}
\usepackage{blkarray}
\usepackage{mathdesign}
\usepackage{color}
\usepackage{geometry}
\usepackage{soul}
\usepackage{nameref}
\mathtoolsset{showonlyrefs}
\usepackage{enumitem}

\definecolor{purple}{rgb}{.9,0,.9}
\definecolor{green}{rgb}{0,.7,0}

\newcommand\subsetsim{\mathrel{\substack{
  \textstyle\subset\\[-0.2ex]\textstyle\sim}}}

\graphicspath{{./figures/}}

\makeatletter
\let\orgdescriptionlabel\descriptionlabel
\renewcommand*{\descriptionlabel}[1]{%
  \let\orglabel\label
  \let\label\@gobble
  \phantomsection
  \edef\@currentlabel{#1}%
  \let\label\orglabel
  \orgdescriptionlabel{#1}%
}
\makeatother

\graphicspath{{./figures/}}

\newcommand{\Nat}{\mathbb{N}}

\newcommand{\Real}{\mathbb{R}}

\newcommand{\cA}{{\cal A}}\newcommand{\cC}{{\cal C}}
\newcommand{\cD}{{\cal D}}\newcommand{\cE}{{\cal E}}\newcommand{\cF}{{\cal F}}
\newcommand{\cG}{{\cal G}}\newcommand{\cH}{{\cal H}}\newcommand{\cI}{{\cal I}}
\newcommand{\cJ}{{\cal J}}
\newcommand{\cM}{{\cal M}}\newcommand{\cN}{{\cal N}}
\newcommand{\cP}{{\cal P}}\newcommand{\cR}{{\cal R}}
\newcommand{\cS}{{\cal S}}\newcommand{\cT}{{\cal T}}\newcommand{\cU}{{\cal U}}
\newcommand{\cV}{{\cal V}}\newcommand{\cW}{{\cal W}}\newcommand{\cX}{{\cal X}}
\newcommand{\cZ}{{\cal Z}}

\newcommand{\ba}{{\bf a}}\newcommand{\bb}{{\bf b}}\newcommand{\bc}{{\bf c}}
\newcommand{\be}{{\bf e}}\newcommand{\bff}{{\bf f}}
\newcommand{\bg}{{\bf g}}\newcommand{\bh}{{\bf h}}

\newcommand{\bm}{{\bf m}}\newcommand{\bn}{{\bf n}}

\newcommand{\bt}{{\bf t}}\newcommand{\bu}{{\bf u}}
\newcommand{\bv}{{\bf v}}\newcommand{\bw}{{\bf w}}\newcommand{\bx}{{\bf x}}
\newcommand{\bA}{{\bf A}}

\newcommand{\bE}{{\bf E}}\newcommand{\bF}{{\bf F}}

\newcommand{\bL}{{\bf L}}\newcommand{\bM}{{\bf M}}
\newcommand{\bP}{{\bf P}}
\newcommand{\bQ}{{\bf Q}}

\newcommand{\ve}{\varepsilon}

\newcommand{\bkappa}{\boldsymbol{\kappa}}
\newcommand{\blambda}{\boldsymbol{\lambda}}

\newcommand{\bbbeta}{\boldsymbol{\beta}}
\newcommand{\bnu}{\boldsymbol{\nu}}
\newcommand{\bmu}{\boldsymbol{\mu}}
\newcommand{\btheta}{\boldsymbol{\theta}}
\newcommand{\bxi}{\boldsymbol{\xi}}
\newcommand{\bomega}{\boldsymbol{\omega}}
\newcommand{\bchi}{{\boldsymbol{\chi}}}

\newcommand{\bphi}{\boldsymbol{\phi}}

\newcommand{\biota}{\boldsymbol{\iota}}

\newtheorem{theorem}{Theorem}[section]
\newtheorem{lemma}[theorem]{Lemma}

\newtheorem{proposition}[theorem]{Proposition}

 \def\smath#1{\text{\scalebox{.8}{$#1$}}}
\def\sfrac#1#2{\smath{\frac{#1}{#2}}}

\newcommand{\beqn}{\begin{equation}}
\newcommand{\eeqn}{\end{equation}}

\newcommand{\bzero}{{\bf 0}}

\newcommand{\trans}{{\scriptscriptstyle\mskip-1mu\top\mskip-2mu}}

\def\H{\mathcal H}
\def\M{\mathcal M}
\def\A{\mathcal A}
\def\W{\mathcal W}

\def\R{\mathbb R}

\def\e{\varepsilon}
\def\s{\sigma}

\def\om{\omega}
\def\bbeta{\boldsymbol{\eta}}
\def\bom{\boldsymbol{\omega}}
\def\l{\lambda}

\def\aa{\textbf{a}}
\def\uu{\textbf{u}}
\def\ee{\textbf{e}}
\def\tt{\textbf{t}}

\def\Om{\Omega}
\def\d{\text{ d}}

\def\spt{{\rm spt}}

\def\pa{\partial}

\def\E{\mathcal{E}}

\def\F{\mathcal{F}}

\newcommand{\vol}{\textbf{\rm \textbf{vol}}}

\renewcommand{\a}{\alpha}
\renewcommand{\b}{\beta}
\renewcommand{\d}{\mathrm{d}}
\newcommand{\D}{\mathcal{D}}

\renewcommand{\l}{\lambda}
\renewcommand{\L}{\Lambda}

\def\tt{\mathbf{t}}
\def\aa{\mathbf{a}}

\numberwithin{equation}{section}

\title{First variation of the fractional $k$-dimensional measure: extending the concept of nonlocal curvature to submanifolds}

\author{Cornelia Mihaila and Brian Seguin}

\begin{document}

\maketitle

\tableofcontents

\begin{abstract}
\noindent The fractional $k$-dimensional measure of a submanifold of $\Real^n$ is a generalization of the fractional perimeter and fractional length appearing in the literature and depends on a parameter $\sigma$ between $0$ and $1$. Here its first variation is computed. The resulting formula is used to define a nonlocal version of the mean-curvature vector for embedded submanifolds. It is shown that in the case where $k=n-1$, this agrees with the nonlocal mean-curvature that has been widely studied.
\end{abstract}

\section{Introduction} 

The nonlocal or fractional perimeter of a bounded set $E\subseteq\R^n$  for $\s\in(0,1)$ was introduced by Caffarelli, Roquejoffre, and Savin in \cite{CRS10}, and is defined by
\beqn
\text{Per}_\sigma(E)\coloneqq\frac{1}{\alpha_{n-1}}\int_E\int_{E^c} |x-y|^{-n-\sigma} dxdy,
\eeqn
where $\alpha_{n-1}$ is the volume of the unit ball in $\Real^{n-1}$. When $E$ is unbounded, the above quantity is generally infinity. Thus, in the unbounded case one needs to consider the fractional perimeter of $E$ relative to a bounded, open set $\Om$:
\begin{align}
\text{Per}_\sigma(E,\Omega)\coloneqq&\frac{1}{\alpha_{n-1}}\Big( \int_{E\cap\Omega}\int_{E^c}+\int_{E\cap\Omega^c}\int_{E^c\cap\Omega}\Big) |x-y|^{-n-\sigma} dxdy\\
\label{sper}=&\frac{1}{\alpha_{n-1}}\int_{E}\int_{E^c} \frac{\max\{\chi_\Omega(x),\chi_\Omega(y)\}}{|x-y|^{n+\sigma}} dxdy.
\end{align}
Caffarelli and Valdinoci \cite{CV11} proved that when $E$ has smooth boundary, for almost every $r>0$,
\beqn\label{limsPer}
\lim_{\sigma\uparrow 1}(1-\sigma)\text{Per}_\sigma(E,B_r)=\cH^{n-1}(\partial E\cap B_r)
\eeqn
which motivates calling it a perimeter function. With this perimeter function, mathematicians have been interested in many of the classical questions from minimal surface theory, including characterizations of minimizers in Plateau's problem or under fixed volume constraints. For instance, Dipierro, Savin, and Valdinoci \cite{DSV17, DSV20} considered a nonlocal Plateau-type problem and considered the irregularities induced by the nonlocal setting, and Cabr\'e, Fall, and Weth \cite{CFSMW16} and Ciraolo, Figalli, Maggi, Novaga \cite{CFMN16} considered properties of critical numbers for constant mean curvature sets with fixed mass. See also the work of Dipierro and Valdinoci \cite{DV99} and the references therein.

By taking the first variation of the nonlocal perimeter functional, Abatangelo and Valdinoci \cite{AV14} derived the nonlocal mean-curvature equation for $E$ which has a $C^{1,\a}$ boundary with $\a>\s$:
\beqn\label{NLMC}
H_\s(z)\coloneqq \frac{1}{\om_{n-2}}\int_{\R^n}\frac{\tilde{\chi}_E(y)\,dx}{|z-y|^{n+\s}},\qquad z\in \pa E,
\eeqn
where $\omega_{n-2}$ is the surface area of the unit ball in $\Real^{n-1}$, $\tilde{\chi}_E=\chi_{E}-\chi_{E^c}$ with $\chi_S$ denoting the characteristic function of a set $S$, and the integral is computed as a principle value. The nonlocal mean-curvature converges to the classical notion in the sense that 
\[
\lim\limits_{\sigma\uparrow 1}(1-\sigma) H_\sigma(z)=H(z).
\]
Moreover, in the case of a point $z\in\pa E$ where most of the nearby points are in $E^c$, like on the sphere, the nonlocal mean curvature is negative. Similarly the nonlocal mean curvature is positive if most of the points near $z$ are in $E$. This is consistent with the sign of the classical mean-curvature when an exterior unit-normal is used.

A limitation of the nonlocal perimeter is that it only describes the perimeter of $(n-1)$-dimensional sets that are the boundaries of $n$-dimensional sets in $\R^n$. Paroni, Podio-Guidugli, and Seguin extended this idea to $(n-1)$-dimensional hypersurfaces in \cite{PPGS99}. By defining $\cX(\partial E)$ to be the set of pairs $(x,y)\in\R^n\times\R^n$ such that the line segment connecting $x$ and $y$ intersects $\partial E$ an odd number of times, they noticed that the definition from Caffarelli, Roquejoffre, and Savin is equivalent to 
\beqn
\text{Per}_\sigma(E,\Omega)=\frac{1}{2\alpha_{n-1}}\int_{\cX(\partial E)} \frac{\max\{\chi_\Omega(x),\chi_\Omega(y)\}}{|x-y|^{n+\sigma}} dxdy.
\eeqn
Indeed, almost every pair of points on opposite sides of the boundary $\partial E$ will cross it an odd number of times. With this perspective they defined 
\beqn\label{sarea}
\text{Area}_\sigma(\cS,\Omega)\coloneqq\frac{1}{2\alpha_{n-1}}\int_{\cX(\cS)} \frac{\max\{\chi_\Omega(x),\chi_\Omega(y)\}}{|x-y|^{n+\sigma}} dxdy
\eeqn
to be the fractional area for a hypersurface $\cS$ in $\R^n$. Under the assumption that $\cS\subseteq\Omega$, they showed that 
\beqn\label{limsarea}
\lim_{\sigma\uparrow 1}(1-\sigma)\text{Area}_\sigma(\cS,\Omega)=\cH^{n-1}(\cS).
\eeqn
Moreover, they computed its first variation and showed that the first variation is zero if and only if
\beqn\label{n-1var}
\lim_{\ve\uparrow 0}\left(\int_{\A_{\rm ext}(z,\e)}-\int_{\A_{\rm int}(z,\e)}\right)\frac{1}{|z-y|^{n+\s}}\,dy=0,\qquad z\in\cS,
\eeqn 
where 
\begin{align*}
\cA_{\rm ext}(z,\e)\coloneqq \{y\in\R^n:&\big((y,z)\in\cX_\e(\cS)\text{ and } (z-y)\cdot \bn(z)>0\big)\\
&\text{or }\big((y,z)\in\cX_\e(\cS)^c\text{ and } (z-y)\cdot \bn(z)<0\big)\},\\
\cA_{\rm int}(z,\e)\coloneqq \{y\in\R^n:&\big((y,z)\in\cX_\e(\cS)^c\text{ and } (z-y)\cdot \bn(z)>0\big)\\
&\text{or }\big((y,z)\in\cX_\e(\cS)\text{ and } (z-y)\cdot \bn(z)<0\big)\},
\end{align*}
with $\cX_\e(\cS)$ being the set of all pairs of points $(y,z)$ such that the segment between $y$ and $z$ crosses $\cS$ an odd number of times and $|z-y|>\e$ and $\bn$ an orientation for $\cS$. This led them to define the nonlocal mean-curvature of $\cS$ at $z$ by
\beqn\label{Hsigma}
H_\s(z)\coloneqq \frac{1}{\om_{n-2}}\int_{\R^n}\frac{\hat{\chi}_\cS(y)}{|z-y|^{n+\sigma}}\,dy,\qquad\text{ for }z\in\cS
\eeqn
where 
\[
\hat{\chi}_\cS(z,y)\coloneqq \begin{cases}
1&y\in\cA_i(z),\\
0& y\notin \cA_i(z)\cup\cA_e(z)\\
-1&y\in\cA_e(z),\\
\end{cases}
\]
and the integral is computed as a principle value. This agrees with \eqref{NLMC} in the case where $\cS$ is the boundary of a set $E$.

In \cite{S20}, Seguin introduced a notion of fractional length. (See Seguin \cite{S20c} for an updated, corrected version of this work.) His motivation came from considering the $n=2$ case, where a curve is a hypersurface. In this setting, a line segment can be viewed as a one-dimensional disk. Thus, rather than considering a pairs of points $(x,y)$ that describe line segments, he used disks described by a center, a unit vector determining the orientation, and a radius. This led to the following definition of fractional length:
\beqn\label{slen}
\text{Len}_\sigma(\cC,\Omega)\coloneqq\int_{\cD(\cC)}  \frac{\sup\{\chi_\Omega(p+r\bv):\bv\in \cU(\{\bu\}^\perp)\}}{r^{1+\sigma}} d\cH^{2n}(p,\bu,r),
\eeqn
where $\cD(\cC)$ consists of all disks, described by a center $p$, unit normal vector $\bu$, and radius $r$, that intersect the curve $\cC$ an odd number of times. Analogous to \eqref{limsarea}, it was shown that
\beqn\label{limslen}
\lim_{\sigma\uparrow 1}(1-\sigma)\text{Len}_\sigma(\cC,\Omega)=\frac{4\pi^{n-1}}{\Gamma(\sfrac{n+1}{2})\Gamma(\sfrac{n-1}{2})(n-1)}\cH^1(\cC),
\eeqn
where $\Gamma$ is the gamma function. The presence of the factor on the right-hand side is due to the fact that the definition of the fractional length does not involve a normalization factor. Assuming that $\cC$ has $C^{1,\a}$-regularity with $\a>\s$, the first variation of $\text{Len}_\s(\cC,\Om)$ vanishing is equivalent to
\[
\lim_{\e\uparrow 0}\Big(\int_{\cA_{\rm o}^+(z,\e)}-\int_{\cA_{\rm e}^+(z,\e)}\Big)\frac{(\bb\wedge\ba)\bt(z)}{r^{1+\s}}\d\H^{2n-2}(\ba,\bb,r)=\textbf{0},\qquad z\in\cC,
\]
where $\bt(z)$ is a unit tangent to $\cC$ at $z$ and
\begin{align*}
\cA^+_{\rm o}(z,\e)&\coloneqq \{(\aa,\bb,r\}\in\cU^2_\perp\times(\e,\infty):\H^0\big(D(z+r\ba,\bb,r)\cap\cC)\text{ is odd}, \bb\cdot\bt(z)>0\},\\
\cA^+_{\rm e}(z,\e)&\coloneqq \{(\aa,\bb,r\}\in\cU^2_\perp\times(\e,\infty):\H^0\big(D(z+r\ba,\bb,r)\cap\cC)\text{ is even}, \bb\cdot\bt(z)>0\}.
\end{align*}
Using this, Seguin defined the nonlocal curvature-vector $\bkappa_\s$ at $z\in\cC$ by
\beqn\label{kapppas}
\bkappa_\s(z)\coloneqq \lim_{\e\rightarrow0}\Big(\int_{\cA_{\rm e}^+(z,\e)}-\int_{\cA_{\rm o}^+(z,\e)}\Big)\frac{(\bb\wedge\ba)\bt(z)}{r^{1+\s}}\d\H^{2n-2}(\ba,\bb,r).
\eeqn 

The interest in understanding fractional measure or mass in higher codimensions has been demonstrated in recent works. Employing the Fourier transform of a measure, Cicalese, Heilmann, Kubin, Onoue, and Ponsiglione \cite{CHKOP25} introduced the notion of $s$-fractional mass for one-currents in higher codimension. Caselli, Freguglia, and Picenni \cite{CFP24} introduced a notion of fractional $s$-mass on ($n-2$)-dimensional closed orientable surfaces in $\R^n$ using the concepts of linking number and fractional Sobolev spaces. In Mihaila and Seguin~\cite{MS25}, the authors extended the notions of fractional length and area for curves and hypersurfaces described above so as to work for any $k$-dimensional manifold in $\Real^n$. They defined the $k$-dimensional $\sigma$-measure of a $k$-dimensional manifold $\cM$ relative to a bounded, open set $\Omega$ by
\begin{equation}\label{smeaskintro}
\text{Meas}_{\s}^k(\M,\Om)\coloneqq \int_{\D(\M)}r^{k-n-\s}  \sup_{\mathbf{u}\in \cU([\bom]^\perp)}\chi_{\Om}(p+r\mathbf{u}) \,d \H^{n+k(n-k)+1}(p,\bom,r),
\end{equation}
where $\cD(\cM)$ is the collection of $(n-k)$-dimensional disks that intersect $\cM$ an odd number of times. In this setting a disk is described by a center point $p$, orientation $\bom$, which is a simple $k$-vector of unit length describing the space orthogonal to the disk, and radius $r$. They proved that if $\M\subseteq \Om$, then 
\begin{equation}\label{limksmeas}
\lim_{\s\uparrow 1} (1-\s)\text{\rm Meas}_{\s}^k(\M,\Om)=\frac{4 \pi^{\frac{(n+2)(k+1)}{2}} \Gamma(\frac{n-k+1}2) }{\sqrt{\pi} \Gamma(\frac{n+1}2)(n-k)}\prod_{i=1}^{k+1}\frac{\Gamma (\frac{i}2)}{\pi^i\Gamma (\frac{n-i+1}2)}\H^k(\M).
\end{equation}
This constant is consistent with the constants in the $k=1$ and $k=n-1$ cases. See \eqref{limslen} and \eqref{limsarea}, respectively. The result holds for all integers $0\le k\le n-1$. 

We use the term nonlocal or fractional measure for $\text{Meas}_\sigma^k$ since it generalizes the concepts of fractional perimeter, fractional area, and fractional length.  However, the $\sigma$-measure relative to a bounded set $\Omega$ is not a measure because it is not additive. This was shown explicitly in \cite{MS25}. In this work, we complete the natural next step of finding necessary conditions for the first variation to be zero and use this to define a nonlocal mean curvature-vector. In particular, we show that a necessary and sufficient condition for the vanishing of the first variation of $\text{\rm Meas}^k_\s(\M,\Om)$ with respect to compact, $k$-dimensional manifolds with the same boundary as $\M$ is that for all $z\in\M$, 

\beqn
\lim_{\e\downarrow 0}\left(\int_{\A_{\rm e}^+(z,\e)}-\int_{\A_{\rm o}^+(z,\e)}\right)\frac{(\aa\wedge\bom)\llcorner\vol(z)}{r^{1+\s}}\,d\H^{w+1}(\aa,\bom,r)=\bzero,
\eeqn
where $w=n-1+k(n-1-k)$, $\vol$ is a volume form for $\cM$, and
\begin{align}
\A_{\rm o}^+(z,\e)&\coloneqq \{(\ba,\bom,r)\in\W\times(\e,\infty): \H^0(D(z+r\aa,\bom,r)\cap \M) \mbox{ is odd},\ \bom\cdot \vol(z)>0\},\\
\A_{\rm e}^+(z,\e)&\coloneqq \{(\ba,\bom,r)\in\W\times(\e,\infty): \H^0(D(z+r\aa,\bom,r)\cap \M) \mbox{ is even},\ \bom\cdot \vol(z)>0\}.
\end{align}
Notice that unlike the nonlocal mean-curvature for hypersurfaces, but similar to the nonlocal curvature for a curve, this quantity is not normalized. Now that the first variation is understood, the next step would be to consider properties of surfaces where the first variation is equal to zero. This work also opens up the discussion of nonlocal mean curvature flow, as in the work of Chambolle, Marini, and Ponsiglione \cite{CMP15}, to a wider variety of surfaces.

Throughout this paper we use $k$-vectors. For a detailed introduction to $k$-vectors, see Ros\'en \cite{R19}.  In an effort to make this work self contained, we have included an introduction to $k$-vectors in the Section~\ref{sec:2}, which overlaps with our set up from \cite{MS25}.  In the next section we also set notion and prove some propositions that will be used in later calculations. In Section~\ref{sec:3} we introduce sets of disks and results about the sets of disks which prepare us the first variation. In Section~\ref{sec:4}, we compute the first variation formula and calculate the mean-curvature vector which works with the definitions from \cite{PPGS99}, \cite{S20}, \cite{S20c}. The Appendix~\ref{secA} provides the change of variables formula and transport theorem needed in the first variation theorem. 

\smallskip
\section{Notation and an introduction to multivectors}\label{sec:2}
\setcounter{section}{2} \setcounter{equation}{0} 

Let $\Real^+\coloneqq(0,\infty)$ denote the set of positive numbers not including zero and $\Real^+_0\coloneqq[0,\infty)$ the set of positive numbers including zero.  Given an inner-product space $\cV$, let $\cU(\cV)$ denote the set of all unit vectors in $\cV$.  If $\cZ$ is a subset of $\cV$, let $\cZ^\perp$ denote those vectors in $\cV$ that are orthogonal to every vector in $\cZ$.   

There are different ways to define multivectors.  Here, we take the approach of viewing them as skew multilinear mappings.  While this is not the definition used in the book by Ros\'en \cite{R19}, it is equivalent, and we recommend Ros\'en's book for an introduction to multivectors. The definition presented here can be found in Federer's book \cite{Fed}. The overview of multivectors presented here is mostly taken from our work \cite{MS25}, with some alterations and additions. 

Fix a natural number $k$ such that $1\leq k\leq n$.  Given $k$ vectors $\bw_i\in\Real^n$, $i\in\{1,\dots,k\}$, define the $k$-vector $\bw_1\wedge\dots\wedge\bw_k$ to be the multilinear map from the $k$-fold product $(\R^n)^k$ to $\R$ by
\beqn\label{simpkvector}
(\bw_1\wedge\dots\wedge\bw_k)(\ba_1,\dots,\ba_k)\coloneqq\text{det}\ 
\begin{blockarray}{cccc}
\begin{block}{(cccc)}
\bw_1\cdot\ba_1&\bw_1\cdot\ba_2&...&\bw_1\cdot\ba_k\\
\bw_2\cdot\ba_1&\bw_2\cdot\ba_2&...&\bw_2\cdot\ba_k\\
\vdots&\vdots&\ddots&\vdots\\
\bw_{k}\cdot\ba_1&\bw_{k}\cdot\ba_2&...&\bw_{k}\cdot\ba_k\\
\end{block}
\end{blockarray}\, ,
\qquad (\ba_1,\dots,\ba_k)\in(\R^n)^k.
\eeqn
Notice that this mapping is skew in the sense that switching any pair of $\ba_i$ results in changing the sign of the output.  The set of all $k$-vectors, denoted by $\Lambda^k(\R^n)$, is the linear span of all linear mappings of the form $\bw_1\wedge\dots\wedge\bw_k$.  It is true that $\text{dim}\,\Lambda^k(\R^n)={n\choose k}=\frac{n!}{k!(n-k)!}$. 

The collection of $k$-vectors of the form \eqref{simpkvector} are called simple $k$-vectors, and are denoted by $\hat\Lambda^k(\Real^n)$.  While $\Lambda^k(\R^n)$ is a vector space, $\hat\Lambda^k(\R^n)$ is not.  However, it is a manifold.  Moreover, simple $k$-vectors span the set of all $k$-vectors.  For a simple $k$-vector $\bomega=\bw_1\wedge\dots\wedge\bw_k$, we denote by $[\bomega]$ the subspace of $\Real^n$ spanned by the $\bw_i$.  Let $\bP_{\bomega}$ and $\bP_{\bomega}^\perp$ be the orthogonal projections of $\R^n$ into $[\bomega]$ and $[\bomega]^\perp$, respectively, so that $\bP_{\bomega}+\bP^\perp_{\bomega} =\textbf{1}_{n}$, where $\textbf{1}_n$ is the identity map on $\Real^n$.

We define several linear mappings on $k$-vectors and pairs of $k$-vectors.  To do so, we first define them on simple $k$-vectors and then extend them to all $k$-vectors by linearity.  The reason that this process works is because $k$-vectors satisfy a universality property.  See Ros\'en \cite{R19} for the details. As a first example of this, consider a linear mapping $\bL:\Real^n\rightarrow \Real^n$.  This induces a linear mapping from $\Lambda^k(\Real^n)$ to itself that we will also denote by $\bL$ such that it acts on a simple $k$-vector $\bu_1\wedge\dots\wedge\bu_k\in\hat\Lambda^k(\R^n)$ by
\beqn\label{Lkvector}
\bL(\bu_1\wedge\dots\wedge\bu_k)\coloneqq\bL\bu_1\wedge\dots\wedge\bL\bu_k.
\eeqn

Given two simple $k$-vectors $\bu_1\wedge\dots\wedge\bu_k,\bv_1\wedge\dots\wedge\bv_k\in\hat\Lambda^k(\R^n)$, their inner-product is defined by
\beqn\label{kvip}
(\bu_1\wedge\dots\wedge\bu_k)\cdot(\bv_1\wedge\dots\wedge\bv_k)\coloneqq\text{det}\ 
\begin{blockarray}{cccc}
\begin{block}{(cccc)}
\bu_1\cdot\bv_1&\bu_1\cdot\bv_2&...&\bu_1\cdot\bv_k\\
\bu_2\cdot\bv_1&\bu_2\cdot\bv_2&...&\bu_2\cdot\bv_k\\
\vdots&\vdots&\ddots&\vdots\\
\bu_{k}\cdot\bv_1&\bu_{k}\cdot\bv_2&...&\bu_{k}\cdot\bv_k\\
\end{block}
\end{blockarray}\, .
\eeqn
Orthogonal, simple $k$-vectors have the following geometric interpretation: given $\bomega,\bnu\in\hat\Lambda^k(\R^n)$ such that $\bomega\cdot\bnu=0$, there is an $\ba\in[\bomega]$ such that $\ba\in[\bnu]^\perp$.

We will denote the set of simple $k$-vectors of unit length by $\hat\Lambda^k_u(\R^n)$.  Elements of $\hat\Lambda^k_u(\R^n)$ can be viewed as oriented $k$-dimensional subspaces of $\R^n$.  Thus, $\hat\Lambda^k_u(\R^n)$ can be used as a double cover for the Grassmannian $\textbf{Gr}(k,\R^n)$.  It follows that $\text{dim}\,\hat\Lambda^k_u(\R^n)=k(n-k)$.  Given a subspace $\cV$ of $\Real^n$, we will use the notation $\hat\Lambda_u^k(\cV)$ for those unit simple $k$-vectors $\bomega$ such that $[\bomega]\subseteq\cV$.

The exterior product $\wedge:\Lambda^k(\R^n)\times\Lambda^\ell(\R^n)\rightarrow\Lambda^{k+\ell}(\R^n)$ is the bilinear map such that given simple multivectors $\bu_1\wedge\dots\wedge\bu_k\in\hat\Lambda^k(\R^n)$ and $\bv_1\wedge\dots\wedge\bv_\ell\in\hat\Lambda^\ell(\R^n)$, we have
\beqn
(\bu_1\wedge\dots\wedge\bu_k)\wedge(\bv_1\wedge\dots\wedge\bv_\ell)\coloneqq\bu_1\wedge\dots\wedge\bu_k\wedge\bv_1\wedge\dots\wedge\bv_\ell.
\eeqn
We will also make use of the interior product of multivectors.  For $k\geq \ell$, the left interior-product $\lrcorner:\Lambda^\ell(\R^n)\times\Lambda^k(\R^n)\rightarrow\Lambda^{k-\ell}(\R^n)$ is the bilinear map such that for $\bnu\in\Lambda^\ell(\R^n)$ and $\blambda\in\Lambda^k(\R^n)$, $\bnu\lrcorner\blambda$ is the unique $(k-\ell)$-vector satisfying
\beqn \label{movingdotL}
\bomega\cdot(\bnu\lrcorner\blambda)=(\bnu\wedge\bomega)\cdot\blambda,\qquad \bomega\in\Lambda^{k-\ell}(\R^n),
\eeqn
and the right interior-product $\llcorner:\Lambda^k(\R^n)\times\Lambda^\ell\R^n)\rightarrow\Lambda^{k-\ell}(\R^n)$ is the bilinear map such that for $\bnu\in\Lambda^\ell(\R^n)$ and $\blambda\in\Lambda^k(\R^n)$, $\blambda\llcorner\bnu$ is the unique $(k-\ell)$-vector satisfying
\beqn \label{movingdotR}
\bomega\cdot(\blambda\llcorner\bnu)=(\bomega\wedge\bnu)\cdot\blambda,\qquad \bomega\in\Lambda^{k-\ell}(\R^n).
\eeqn
These two interior products differ at most by a sign in that
\beqn\label{LRswitch}
\bnu\lrcorner\blambda=(-1)^{\ell(k-\ell)}\blambda\llcorner\bnu,\quad \bnu\in\Lambda^\ell(\R^n),\ \blambda\in\Lambda^k(\R^n).
\eeqn

It transpires that if $\bom\in\hat\Lambda^k_u(\Real^n)$, then
\beqn\label{projip}
\bP_{\bom}\bu=(\bom\llcorner\bu)\lrcorner\bom,\qquad \bu\in\Real^n.
\eeqn
Given an orientation for $\Real^n$ in the form of a top dimensional multivector $\biota\in\hat\Lambda_u^n(\Real^n)$, one can then define the (left) Hodge star map $*:\Lambda^k(\Real^n)\rightarrow\Lambda^{n-k}(\Real^n)$ by
\beqn\label{Hstar}
*\bomega\coloneqq \biota\llcorner \bomega,\qquad \bomega\in\Lambda^k(\Real^n).
\eeqn
The Hodge star is an isometry on multvectors and satisfies the identity
\beqn\label{IPstar}
*(\bnu\lrcorner\blambda)=(* \blambda)\wedge\bnu, \quad \bnu\in\Lambda^\ell(\R^n),\ \blambda\in\Lambda^k(\R^n).
\eeqn
A useful relation involving the left interior-product is
\beqn\label{moveuk}
(\bomega\wedge\blambda)\lrcorner\bmu=\blambda\lrcorner(\bomega\lrcorner\bmu)
\eeqn
for all $\bomega\in\Lambda^{k}(\R^n)$, $\blambda\in\Lambda^{\ell}(\R^n)$, and $\bmu\in\Lambda^{m}(\R^n)$ where $k+\ell\leq m$.  There is also a varient on Lagrange's identity: if $\ba\in\Real^n$ and $\bomega\in\Lambda^k(\Real^n)$, then
\beqn\label{Lagrange}
|\ba|^2|\bomega|^2=|\ba\lrcorner\bomega|^2+|\ba\wedge\bomega|^2.
\eeqn
There is also the anticommutation relation:
\beqn\label{anticom}
\ba\lrcorner(\bb\wedge\bomega)+\bb\wedge(\ba\lrcorner\bomega)=(\ba\cdot\bb)\bomega,
\eeqn
which holds for all $\ba,\bb\in\Real^n$ and $\bomega\in\Lambda^k(\Real^n)$. It can be shown that for a simple $k$-vector $\bomega$,
\beqn
[\bomega]=\{\ba\in\R^n\ :\ \ba\wedge\bomega=\textbf{0}\}\quad\text{and}\quad [\bomega]^\perp=\{\ba\in\R^n\ :\ \ba\lrcorner\bomega=\textbf{0}\}.
\eeqn
Analogous facts are true involving the right interior-product.

We will illustrate the left interior-product by mentioning a fact that will be of use later.  Let $\bomega\in\hat\Lambda_u^4(\Real^n)$.  It follows that there are orthonormal vectors $\bw_1,\bw_2,\bw_3$, and $\bw_4$ such that $\bomega=\bw_1\wedge\bw_2\wedge\bw_3\wedge\bw_4$.  From the definitions above, for $\ba\in\Real^n$,
\beqn
\ba\lrcorner \bomega= (\ba\cdot\bw_1)\bu_2\wedge\bw_3\wedge\bw_4-(\ba\cdot\bw_2)\bw_1\wedge\bu_3\wedge\bw_4+(\ba\cdot\bw_3)\bw_1\wedge\bw_2\wedge\bw_4-(\ba\cdot\bw_4)\bw_1\wedge\bw_2\wedge\bw_3.
\eeqn
Thus,
\beqn
\bw_2\lrcorner \bomega= -\bw_1\wedge\bw_3\wedge\bw_4,
\eeqn
and for $\bb\in\Real^n$,
\beqn
\bb\wedge(\bw_2\lrcorner\bomega)=\bw_1\wedge\bb\wedge\bw_3\wedge\bw_4.
\eeqn
This kind of identity generalizes.  Namely, if $\bomega=\bw_1\wedge\dots\wedge\bw_k\in\hat\Lambda_u^k(\Real^n)$ and $\bb\in\Real^n$, then $\bb\wedge(\bw_i\lrcorner\bomega)$ is the $k$-vector obtained by replacing $\bw_i$ with $\bb$ in $\bomega$---that is,
\beqn\label{replacewedgeID}
\bb\wedge(\bw_i\lrcorner\bomega)=\bw_1\wedge\dots\wedge\bw_{i-1}\wedge\bb\wedge\bw_{i+1}\wedge\dots\bw_k.
\eeqn

We now establish several facts involving multivectors that will be of use in this work that are not present in Ros\'en \cite{R19}.

\begin{proposition}\label{dotreduction} Let $\aa\in\cU(\Real^n)$ and $\bom,\bmu\in\hat\Lambda^k(\Real^n)$. If $\aa\in[\bom]^\perp$, then 
\[
(\aa\wedge\bom)\cdot(\aa\wedge\bmu)=\bom\wedge\bmu.
\]
\end{proposition}
\begin{proof}
As $\bom$ and $\bmu$ are simple $k$-vectors, they can be written as $\bom=\bw_1\wedge...\wedge\bw_k$ and $\bmu=\bm_1\wedge...\wedge \bm_k$, where $\{\bw_1,\dots,\bw_k\}$ and $\{\bm_1,\dots\bm_k\}$ are orthonormal basis for $[\bom]$ and $[\bmu]$, respectively. From the definition of the inner-product of multivectors \eqref{kvip}, we have 
\[(\aa\wedge\bom)\cdot(\aa\wedge\bmu)=
\det\begin{pmatrix}
\aa\cdot \aa& \aa\cdot \bm_1&\hdots&\aa\cdot \bm_k \\
\bw_1\cdot \aa &\bw_1\cdot\bm_1&\hdots&\bw_1\cdot\bm_k\\
\vdots &\vdots&\ddots&\vdots\\
\bw_k\cdot\aa&\bw_k\cdot\bm_1&\hdots&\bw_k\cdot\bm_k\,
\end{pmatrix}.\]
Since $\aa\in[\bom]^\perp$, we know $\bw_i\cdot \aa=0$ for $1\le i\le k$, and because $\aa$ is a unit vector, it follows that 
\[(\aa\wedge\bom)\cdot(\aa\wedge\bmu)=
\det\begin{pmatrix}
1& \aa\cdot \bm_1&\hdots&\aa\cdot \bm_k \\
0&\bw_1\cdot\bm_1&\hdots&\bw_1\cdot\bm_k\\
\vdots &\vdots&\ddots&\vdots\\
0&\bw_k\cdot\bm_1&\hdots&\bw_k\cdot\bm_k\,
\end{pmatrix}.\]
Computing the determinant using minors gives the desired result. 
\end{proof}

\begin{proposition}\label{tusimpmv}
Consider $\bom\in\hat\Lambda_u^k(\Real^n)$ and $\ba\in[\bom]^\perp$. For any $\bnu\in\Lambda^{k-1}([\bom])$, we have
\beqn
\ba\wedge\bnu\in T_{\bom} \hat\Lambda_u^k(\Real^n),
\eeqn
meaning that $\ba\wedge\bnu$ is in the tangent space of the manifold $\hat\Lambda_u^k(\Real^n)$ at $\bom$.
\end{proposition}

\begin{proof}
It suffices to prove the result for a collection of $\bnu$ that span $\Lambda^{k-1}([\bom])$. Thus, we need only consider the case where $\bnu\in\hat\Lambda_u^{k-1}([\bom])$, so that $\bnu$ is a simple $(k-1)$-vector of unit magnitude. Similarly, we may assume that $\ba$ is a unit vector. As $\bom$ and $\bnu$ are both simple multivectors and $[\bnu]\subseteq[\bom]$, there is a $\bb\in[\bnu]^\perp$ such that $\bom=\bb\wedge\bnu$. Moreover, as $\bom$ and $\bnu$ have unit magnitude, $\bb$ must have unit magnitude. Consider the curve $\hat\ba:\Real\rightarrow\Real^n$ defined by $\hat\ba(s)\coloneqq \cos(s)\bb+\sin(s)\ba$ and the curve $\hat\bom:\Real\rightarrow \Lambda^k(\Real^n)$ defined by $\hat\bom(s)=\hat\ba(s)\wedge\bnu$. One can readily check that $|\hat\bom(s)|=1$ and $\hat\bom(s)$ is simple for all $s\in\Real$, so $\hat\bom$ is a curve in $\hat\Lambda^{k}_u(\Real^n)$. As $\hat\bom(0)=\bom$, it follows that
\beqn
\hat\bom'(0)=\ba\wedge\bnu\in T_{\bom} \hat\Lambda_u^k(\Real^n),
\eeqn
which establishes the claim.
\end{proof}

\begin{proposition}\label{magproj}
If $\bmu\in\hat\Lambda^k(\Real^n)$ and $\bom\in\hat\Lambda^\ell(\Real^n)$, with $k\leq \ell$, then $|\bmu\lrcorner\bom|=|\bom||\bP_{\bom}\bmu|$.
\end{proposition}

\begin{proof}
The proof is by induction on $k$. The $k=1$ case was established by Mihaila and Seguin \cite{MS25}[Lemma~1]. Now assume that 
\beqn\label{Ihypoth}
|\blambda\lrcorner\bom|=|\bom||\bP_{\bom}\blambda|\quad\text{for}\ \blambda\in\hat\Lambda^{k-1}(\Real^n).
\eeqn
Given $\bmu\in\hat\Lambda^k(\Real^n)$, find $\ba\in\Real^n$ and $\blambda\in\hat\Lambda^{k-1}(\Real^n)$ such that
\beqn\label{inducdecomp}
\ba\lrcorner\blambda=\bzero\quad\text{and}\quad \bmu=\ba\wedge\blambda.
\eeqn
It follows from \eqref{moveuk}, the induction hypothesis \eqref{Ihypoth}, \eqref{inducdecomp}, and Lagrange's identity \eqref{Lagrange} that
\begin{align*}
|\bmu\lrcorner\bomega|^2&=|(\ba\wedge\blambda)\lrcorner\bom|^2\\
&=|\blambda\lrcorner(\ba\lrcorner\bom)|^2\\
&=|\ba\lrcorner\bom|^2|\bP_{\ba\lrcorner\bom}\blambda|^2\\
&=|\bom|^2||\bP_{\bom}\ba|^2|\bP_{\ba\lrcorner\bom}\blambda|^2\\
&=|\bom|^2[|\bP_{\bom}\ba\wedge \bP_{\ba\lrcorner\bom}\blambda|^2+|\bP_{\bom}\ba\lrcorner \bP_{\ba\lrcorner\bom}\blambda|^2].
\end{align*}
One can find an orthonormal basis $\{\bw_1,\dots,\bw_\ell\}$ of $[\bom]$ such that $\bom=\bw_1\wedge\dots\wedge\bw_\ell$ and $\ba\cdot\bw_i=0$ for $i>1$. It follows that
\beqn
\bP_{\bom}\ba=(\ba\cdot\bw_1)\bw_1\quad\text{and}\quad \bP_{\ba\lrcorner\bom}\blambda\in\hat\Lambda^{k-1}([\bw_2\wedge\dots\wedge\bw_k]).
\eeqn
Thus, $\bP_{\bom}\ba\lrcorner \bP_{\ba\lrcorner\bom}\blambda=\bzero$. Moreover, from \eqref{inducdecomp}$_1$,
\beqn
\bP_{\ba\lrcorner\bom}\blambda=\bP_{\bom}\blambda.
\eeqn
This means that $\bP_{\bom}\ba\wedge \bP_{\ba\lrcorner\bom}\blambda=\bP_{\bom}\ba\wedge\bP_{\bom}\blambda=\bP_{\bom}(\ba\wedge\blambda)$, which completes the proof.
\end{proof}

\begin{proposition}\label{wedgenorm}
For $\ba\in\cU(\Real^n)$, $\bom\in\hat\Lambda^k_u(\Real^n)$ with $\ba\lrcorner\bom=\bzero$, and $\bmu\in\hat\Lambda^k_u(\Real^n)$, it is true that
\beqn
|(\ba\wedge\bom)\llcorner\bmu|^2=(\bom\cdot\bmu)^2+|\bP_{\bom}(\ba\lrcorner\bmu)|^2.
\eeqn
\end{proposition}

\begin{proof}
Notice that for any $\bu\in\Real^n$, by the definition of the right and left interior-products and \eqref{anticom},
\begin{align}
(\ba\wedge\bom)\llcorner\bmu\cdot\bu&=\ba\wedge\bom\cdot\bu\wedge\bmu\\
&=\bu\lrcorner(\ba\wedge\bom)\cdot\bmu\\
&=[(\bu\cdot\ba)\bom-\ba\wedge(\bu\lrcorner\bom)]\cdot\bmu\\
&=(\bom\cdot\bmu)(\ba\cdot\bu)-\bom\cdot\bu\wedge(\ba\lrcorner\bmu)\\
&=[(\bom\cdot\bmu)\ba-\bom\llcorner(\ba\lrcorner\bmu)]\cdot\bu.
\end{align}
It follows that
\beqn\label{magnormid}
(\ba\wedge\bom)\llcorner\bmu=(\bom\cdot\bmu)\ba-\bom\llcorner(\ba\lrcorner\bmu).
\eeqn
Notice that the two terms on the right-hand side of \eqref{magnormid} are orthogonal since from the assumption $\ba\lrcorner\bom=\bzero$,
\begin{align}
\bom\llcorner(\ba\lrcorner\bmu)\cdot\ba&=\bom\cdot\ba\wedge(\ba\lrcorner\bmu)\\
&=\ba\lrcorner\bom\cdot\ba\lrcorner\bmu\\
&=0.
\end{align}
It then follows from Proposition~\ref{magproj} and \eqref{LRswitch} that
\beqn
|(\ba\wedge\bom)\llcorner\bmu|^2=(\bom\cdot\bmu)^2+|\bom\llcorner(\ba\lrcorner\bmu)|^2=(\bom\cdot\bmu)^2+|\bP_{\bom}(\ba\lrcorner\bmu)|^2.
\eeqn
\end{proof}

\smallskip
\section{Set of disks}\label{sec:3}
\setcounter{section}{3} \setcounter{equation}{0} 

In this section we discuss the set of $(n-k)$-dimensional disks in $\Real^n$ used to define the fractional $k$-dimensional measure. Moreover, several measure theoretic properties of these disks will be established that will be of use in this work.  We find it convenient to work with oriented disks, so for the rest of the paper a disk refers to an oriented disk.

Each $(n-k)$-dimensional disk can be represented by its center $p\in\Real^n$, radius $r>0$, and perpendicular $k$-dimensional (oriented) subspace represented by $\bom\in \hat{\Lambda}^k_u(\R^{n})$. The disk with center $p$, radius $r$, and orientation $\bomega$ is denoted by
\[
D(p,\bom, r)\coloneqq\{p+\xi \mathbf{v}\in\R^n:\mathbf{v}\in\cU([\bomega]^\perp), \xi\in [0,r)\}.
\]
Thus, the set of all disks can be described by
\[
\D\coloneqq\R^n\times \hat{\Lambda}^k_u(\R^{n})\times \R^+.
\]
The dimension of this set is denoted by $d\coloneqq \text{dim}\,\cD=n+k(n-k)+1$. By the boundary $\pa D(p,\bom, r)$ of a disk $D(p,\bom,r)$ we mean the $(n-k-1)$-dimensional manifold 
\[
\pa D(p,\bom, r)\coloneqq\{p+r \mathbf{v}\in\R^n: \mathbf{v}\in\cU([\bomega]^\perp)\}.
\]
We denote by $\overline{D}(p,\bom,r)$ the closure of the disk $D(p,\bom,r)$.

Consider a $k$-dimensional $C^1$ manifold with boundary, where $1\leq k\leq n-1$, which we denote by $\overline\cM$. $\cM$ will denote the interior of the manifold and $\partial\cM$ its boundary, so that $\overline\cM=\cM\cup\partial\cM$. Let $\vol$ denote a volume form for $\overline\cM$, meaning that for each $z\in\overline\cM$, $\vol(z)\in\hat\Lambda^k_u(\Real^n)$ and $[\vol(z)]=T_z\overline\cM$.  It is not assumed that $\overline\cM$ is orientable, so it is possible that $\vol$ is not continuous.  Define the following subsets of $\D$: 
\begin{align*}
\D_{\pa \M}&\coloneqq\{(p,\bom,r)\in\D:\H^0(\overline{D}(p,\bom,r)\cap \pa \M)\neq 0 \},\\
\D_{\tan}&\coloneqq\{(p,\bom,r)\in \D: \text{there is } z\in \overline{D}(p,\bom,r)\cap \overline{\M}\text{ such that } \bomega\cdot\vol(z)=0\},\\
\D_\infty&\coloneqq\{(p,\bom,r)\in \D: \H^0(\overline{D}(p,\bom,r)\cap \overline{\M})=\infty\},\\ 
\D_{\pa D}&\coloneqq\{(p,\bom,r)\in \D: \H^0(\pa D(p,\bom,r)\cap \M)\neq0\}.
\end{align*}
Roughly speaking, the interpretation of these sets is as follows: $\cD_{\partial \cM}$ consists of those disks whose closure intersect $\partial\cM$, $\cD_\text{tan}$ consists of those disks that are tangent---that is, not transversal---to the manifold, $\cD_\infty$ is the collection of disks that intersect $\overline\cM$ an infinite number of times, and $\cD_{\partial D}$ consists of those disks whose boundary intersects $\cM$. 

Mihaila and Seguin \cite{MS25}[Lemma~2] established the following facts about these sets.
\begin{lemma}\label{lemmeasure}
Given an open, bounded set $\cE\subseteq\cD$, the following results hold:
\begin{enumerate}
\item $\H^{d-1}(\cD_{\pa \M}\cap\E)<\infty$,\label{paM}
\item $\H^{d-1}(\cD_{\rm{tan}}\cap\E)<\infty$,\label{tan}
\item$\H^{d-1}(\cD_{\pa D}\cap\E)<\infty$,\label{paD} 
\item$\D_\infty\subseteq\D_{\text{\rm tan}}$.\label{inftan} 
\end{enumerate}
\end{lemma}
In this work we require several additional subsets of disks and results about them.  Namely, consider the sets
\begin{align}
\D_{\pa \M 1}&\coloneqq\{(p,\bom,r)\in\D:\H^0(\overline{D}(p,\bom,r)\cap \pa \M)=1 \},\\
\D_{\pa \M 2}&\coloneqq\{(p,\bom,r)\in\D:\H^0(\overline{D}(p,\bom,r)\cap \pa \M)\geq 2 \},\\
\D_{\pa D 1}&\coloneqq\{(p,\bom,r)\in \D: \H^0(\pa D(p,\bom,r)\cap \M)= 1\},\\
\D_{\pa D 2}&\coloneqq\{(p,\bom,r)\in \D: \H^0(\pa D(p,\bom,r)\cap \M)\geq 2\},\\
\label{Dodd}\D_{\rm odd}&\coloneqq\{(p,\bom,r)\in \D\setminus(\cD_{\partial\cM}\cup\cD_{\partial D}\cup\cD_{\rm{tan}})  : \H^0( D(p,\bom,r)\cap \M)\ \text{is odd}\},\\
\D_{\rm even}&\coloneqq\{(p,\bom,r)\in \D\setminus(\cD_{\partial\cM}\cup\cD_{\partial D}\cup\cD_{\rm{tan}}): \H^0( D(p,\bom,r)\cap \M)\ \text{is even}\}.
\end{align}
The set $\cD_{\rm odd}$ is exactly the set $\cD(\cM)$ used to define the fractional $k$-dimensional measure; see \eqref{smeaskintro}. Notice that $\cD_{\partial\cM}=\cD_{\partial\cM1}\cup\cD_{\partial\cM2}$ and $\cD_{\partial D}=\cD_{\partial D1}\cup\cD_{\partial D2}$. We will make use of the set
\beqn\label{Wdef}
\W\coloneqq\{(\mathbf{a},\bom)\in \cU(\R^n)\times \hat{\Lambda}^k_u(\R^{n}): \mathbf{a}\in [\bom]^\perp\}.
\eeqn
Each element in $\W$ gives a direction $\mathbf{a}$ and a unit magnitude simple $k$-vector $\bom$, which determines a space that is perpendicular to $\mathbf{a}$.  Put another way, $\bom$ is a $k$-vector in the $(n-1)$-dimensional space $[\ba]^\perp$.  Therefore, the dimension of $\W$ is $w\coloneqq \text{dim}\,\cW=n-1+k(n-1-k)$.

\begin{lemma}\label{lemmeasure2}
The following results hold:
\begin{enumerate}
\item $\H^{d-1}(\D_{\pa D2 })=0$,\label{paD2}
\item $\H^{d-1}(\D_{\pa \cM 2})=0$,\label{paM2}
\item $\H^{d-1}(\D_{\pa \M}\cap \D_{\pa D})=0,$\label{paMD} 
\item $\D_{\rm{odd}}$ and $\D_{\rm{even}}$ are open subsets of $\D$, \label{eo} 
\item $\D=\D_{\rm{odd}}\cup \D_{\rm{even}}\cup \D_{\pa \M}\cup \D_{\rm tan}\cup \D_{\pa D}$.\label{Dsep}
\end{enumerate}
\end{lemma}
\begin{proof}

\textit{Item 1.} To get the control on the measure of $\cD_{\pa D2}$, we subdivide it into two subsets. Define
\begin{align*}
S_1&\coloneqq \{(p,\bom,r)\in \cD_{\pa D2}: p=\tfrac{z_1+z_2}{2} \mbox{ for some } z_1,z_2\in \cM \cap\pa D(p,\bom,r)\},\\
S_2&\coloneqq \cD_{\pa D2}\setminus S_1.
\end{align*}
The set $S_1$ consists of those disks in $\cD_{\pa D2}$ with center the midpoint between two points of intersection of $\pa D(p,\bom,r)$ and $ \cM$, and $S_2$ is the set of disks in $\cD_{\pa D2}$ with two points of $\M$ in $\pa D(p,\bom,r)$ but with center not on the line segment connecting those two points. 

Consider the set
\[
A_{\pa D2}\coloneqq\{(z_1,z_2,\bom)\in \M\times \M\times \hat{\Lambda}^k_u(\R^n): (z_2-z_1)\lrcorner\bom =\bf0\},
\]
and define $\Phi_{\pa D2}:A_{\pa D2}\rightarrow \D$ by 
\[
\Phi_{\pa D2}(z_1,z_2,\bom)\coloneqq\left(\tfrac{1}2(z_1+z_2),\bom, \tfrac{1}2|z_2-z_1|\right).
\]
Notice that $S_1\subseteq \Phi_{\pa D2}(A_{\pa D2})$. Since the map $\Phi_{\pa D2}$ is locally Lipschitz, it follows that the Hausdorff dimension of $S_1$ is at most equal to the Hausdorff dimension of $A_{\pa D2}$. The dimension of $A_{\pa D2}$ comes down to choosing $z_1,z_2\in\M$, which has dimension $k$, and $\bom\in\hat\Lambda_u^k([z_2-z_2]^\perp)$, which has dimension $k(n-1-k)$. So, the Hausdorff dimension of $A_{\pa D2}$ is $k+k(n-k)$, which is strictly less than $n+k(n-k)=d-1$.

Now consider the set
\[
A'_{\pa D2}\coloneqq\{(z_1,z_2,\mathbf{a},\bom)\in \M\times \M\times \W: \ba\cdot (z_2-z_1)>0, (z_2-z_1)\lrcorner\bom =\bf0\},
\]
and define the function $\Phi'_{\pa D2}:A_{\pa D2}\rightarrow \D$ by 
\[
\Phi'_{\pa D2}(z_1,z_2,\mathbf{a},\bom)\coloneqq\left(z_1+\frac{|z_2-z_1|^2}{2\mathbf{a}\cdot(z_2-z_1)}\mathbf{a}, \bom, \frac{|z_2-z_1|^2}{2\mathbf{a}\cdot(z_2-z_1)}\right).
\]
Then $S_2\subseteq \Phi'_{\pa D2}(A'_{\pa D2})$. Since $\Phi'_{\pa D2}$ is locally Lipschitz, to show that $\H^{d-1}(S_2)=0,$ it suffices to show that the Hausdorff dimension of $A_{\pa D2}$ is $d-2=(n-1)+k(n-k)$. Since $\aa$ is a unit vector in $\Real^n$, it is chosen from an $(n-1)$-dimensional set. Since $\mathbf{a}\in [\bom]^\perp$ and $(z_2-z_1)\lrcorner\bom=\bf0$, and because $\ba$ is not parallel to $z_2-z_1$ by construction, given $z_1$, $z_2$, and $\aa$, $\bom$ is choosen from a set of dimension $k(n-2-k)$. Summing the dimensions of the choices for $z_1,$ $z_2$, $\aa$, and $\bom$, yields $d-2.$ 

Since $\D_{\pa D 2}\subseteq \Phi_{\pa D 2}(A_{\pa D2})\cup \Phi'_{\pa D 2}(A'_{\pa D2})$, it follows that $\H^{d-1}(D_{\pa D2})=0$.

\medskip
\noindent\textit{Item 2.} We proceed as in the previous item. Subdivide $\D_{\pa \cM 2}$ into two sets:
\begin{align*}
T_1&\coloneqq \{(p,\bom,r)\in \cD_{\pa \M2}: p=z_1+\xi\tfrac{z_2-z_1}{|z_2-z_1|}\mbox{ for some } z_1,z_2\in \pa \cM \cap\overline{D}(p,\bom,r),\xi\in\Real^+_0\},\\
T_2&\coloneqq \cD_{\pa \M2}\setminus T_1.
\end{align*}
Notice that $T_1$ consists of those disks that have a center collinear with two points in $\pa \cM \cap\overline{D}(p,\bom,r)$, and $T_2$ is the set of disks with two points in $\pa \cM \cap\overline{D}(p,\bom,r)$ with center not between two such points. 

Consider the set
\[
A_{\pa \M2}\coloneqq\{(z_1,z_2,\bom, \xi,r)\in\pa \M\times\pa \M\times \hat{\Lambda}^k_u(\R^n)\times\R_0^+\times\R^+: (z_2-z_1)\lrcorner\bom={\bf0},\xi\le r\},
\]
and define $\Phi_{\pa \M2}:A_{\pa \M2}\rightarrow \D$ by 
\[
\Phi_{\pa \M2}(z_1,z_2,\bom,\xi,r)\coloneqq\left(z_1+\xi\frac{z_2-z_1}{|z_2-z_1|},\bom, r\right).
\]
Notice that $T_1\subseteq \Phi_{\pa \M2}(A_{\pa \M2})$ and this mapping is locally Lipschitz. The dimension of $A_{\pa \M2}$ comes from choosing $z_1,z_1\in \pa \M$, $\bom\in\hat\Lambda_u^k([z_2-z_2]^\perp)$, and real numbers $\xi$ and $r$, so the dimension of $A_{\pa M2}$ is $k+k(n-k)$.

Now consider the set
\[
A'_{\pa \M2}\coloneqq\{(z_1,z_2,\mathbf{a},\bom,\xi,r)\in \pa\M\times \pa\M\times \W\times \R_0^+\times\R^+: \ba\cdot (z_2-z_1)>0, (z_2-z_1)\lrcorner\bom ={\bf0},\xi\le r\},
\]
and define the function $\Phi'_{\pa \M2}:A'_{\pa \M2}\rightarrow \D$ by 
\[
\Phi'_{\pa \M2}(z_1,z_2,\mathbf{a},\bom)\coloneqq\left(z_1+\xi \mathbf{a}, \bom, r\right).
\]
Then $T_2\subseteq\Phi'_{\pa \M2}(A'_{\pa \M2})$. Since the map $\Phi'_{\pa \M2}$ is locally Lipschitz, to show that $\H^{d-1}(T_2)=0,$ it suffices to argue that the Hausdorff dimension of $A'_{\pa \M2}$ is $d-2=(n-1)+k(n-k)$. Notice that, since $\aa$ is a unit vector in $\Real^n$, it is chosen from an $(n-1)$-dimensional set. Since we need $\mathbf{a}\in [\bom]^\perp$ and $(z_2-z_1)\lrcorner\bom=\bzero$, the dimension of the set to choose $\bom$ from given $z_1$, $z_2$, and $\ba$ is $k(n-2-k)$. Summing the dimensions of the choices for $z_1,$ $z_2$, $\aa$, $\bom$, $\xi$, and $r$, we get $d-2.$ 

Since $\D_{\pa \M 2}\subseteq \Phi_{\pa \M2}(A_{\pa \M2})\cup \Phi'_{\pa \M2}(A'_{\pa \M2})$, it follows that $\H^{d-1}(D_{\pa \M2})=0$.

\medskip
\noindent\textit{Item 3.} Consider the sets
\begin{align*}
R_1&\coloneqq \{(p,\bom,r)\in\D_{\pa \M}\cap D_{\pa D}: p=z_2+r\tfrac{z_1-z_2}{|z_1-z_2|}\\
&\qquad\qquad\mbox{ for some } z_1\in \pa \cM\cap\overline{D}(p,\bom,r)\ \text{and}\ z_2\in \cM \cap\pa D(p,\bom,r)\},\\
R_2&\coloneqq (\D_{\pa \M}\cap D_{\pa D})\setminus R_1.
\end{align*}
The set $R_1$ is the disks in $\D_{\pa \M}\cap D_{\pa D}$ with center collinear with a point in $\pa \cM\cap\overline{D}(p,\bom,r)$ and a point in $\cM \cap\pa D(p,\bom,r)$.
Define 
\[
A_{\cap 1}\coloneqq \{(z_1,z_2,\bom,r)\in \pa \M\times \M\times \hat{\Lambda}^k_u(\R^n) \times \R^+:  (z_1-z_2)\lrcorner\bom=\bzero \},
\]
and 
\[
A_{\cap 2}\coloneqq \{(z_1,z_2,\aa,\bom,r\in \pa \M\times \M\times \W\times \R^+: \ba\cdot (z_1-z_2)>0, (z_1-z_2)\lrcorner\bom=\bzero \}.
\]
We can define $\Phi_{\cap 1}:A_{\cap 1}\rightarrow \D$ by 
\[
\Phi_{\cap 1}(z_1,z_2,\bom,r)\coloneqq \big(z_2+r\tfrac{z_1-z_2}{|z_1-z_2|}, \bom, r\big)\,.
\]
and $\Phi_{\cap 2}:A_{\cap 2}\rightarrow \D$
\[
\Phi_{\cap 2}(z_1,z_2,\ba,\bom,r)\coloneqq\left(z_2+r\ba, \bom, r\right)\,.
\]
Using arguments as before, one finds that the dimension of $A_{\cap 1}$ is $k+k(n-k)$ and the dimension of $A_{\cap 2}$ is $n-1+k(n-k)$. So, analogous to the previous items, it follows that $\H^{d-1}(D_{\pa \M}\cap D_{\pa D})=0.$

\medskip
\noindent\textit{Items 4 and 5.} Both follow from the definitions of the sets involved. 
\end{proof}

Next we show that two special collections of disks are $(d-1)$-dimensional submanifolds of $\cD$ and characterize the tangent spaces of one of them.

\begin{lemma}\label{lempDemb}
The set $\cD_{\partial D1}\setminus(\cD_{\partial \cM}\cup\cD_\text{\rm tan})$ is an embedded submanifold of $\cD$ of dimension $d-1$, and given $(p,\bom,r)$ in this submanifold, with $z\in \partial D(p,\bom,r)\cap \cM$ and setting $\ba\coloneqq (p-z)/r$, the tangent space
\beqn
T_{(p,\bom,r)}(\cD_{\partial D1}\setminus(\cD_{\partial \cM}\cup\cD_\text{\rm tan}))
\eeqn
is spanned by vectors of the form:
\begin{itemize}
\item $k$ dimensions: $(\bt,\bzero,0)$ for $\bt\in T_z\cM$,
\item $n-k-1$ dimensions: $(\bc,\bzero,0)$ for $\bc\in[\bom]^\perp$ such that $\bc\cdot\ba=0$,
\item $k(n-k)-k$ dimensions: $(\bzero,\bbeta,0)$ for $\bbeta\in T_{\bom} \hat\Lambda^k_u(\Real^n)$ such that $\ba\lrcorner \bbeta=\bzero$,
\item $k$ dimensions: $(r\bb,-\ba\wedge(\bb\lrcorner\bom),0)$ for $\bb\in[\bom]$,
\item 1 dimension: $(\ba,0,1)$.
\end{itemize}
Moreover, the set $\cD_{\partial \M1}\setminus(\cD_{\partial D}\cup\cD_\text{\rm tan})$ is also an embedded submanifold of $\cD$ of dimension $d-1$.
\end{lemma}

\begin{proof}
First we establish the results about the set $\cD_{\partial D1}\setminus(\cD_{\partial \cM}\cup\cD_\text{\rm tan})$. Consider the function $\Psi:\overline\cM\times\cW\times\Real^+\rightarrow\D$ defined by
\beqn
\Psi(z,\ba,\bom,r)\coloneqq (z+r\ba,\bom,r),\qquad (z,\ba,\bom,r)\in\overline\cM\times\cW\times\Real^+.
\eeqn
Notice that by the definition of this function, $\Psi(\cM\times\cW\times\Real^+)=\cD_{\partial D}$. Now consider the set 
\beqn
\hat\cA_{\partial D1}\coloneqq \Psi^{-1}(\cD_{\partial D1}\setminus(\cD_{\partial\cM}\cup\cD_{\rm tan}))\cap(\cM\times\cW\times\Real^+).
\eeqn
As $\cM\times\cW\times\Real^+$ is a manifold of dimension $d-1$, we will show that $\hat\cA_{\partial D1}$ is an open subset of $\cM\times\cW\times\Real^+$ to prove that it is also a manifold of the same dimension. Begin by fixing $(z,\ba,\bom,r)\in\hat\cA_{\partial D1}$. Since $\cD_{\partial\cM}$ and $\cD_{\rm tan}$ are closed subsets of $\cD$ and $\Psi$ is smooth, one can choose a small enough open neighborhood of $(z,\ba,\bom,r)$ such that the image of every element of this neighborhood under $\Psi$ is not in $\cD_{\partial\cM}$ nor in $\cD_{\rm tan}$. Moreover, this neighborhood can be chosen small enough so that its image under $\Psi$ is contained in $\cD_{\partial D1}.$ To see this, suppose that it cannot be chosen disjoint from $\cD_{\partial D2}$. Then there is a sequence $(z_n,\ba_n,\bom_n,r_n)\in\hat\cA_{\partial D1}$ that converges to $(z,\ba,\bom,r)$ such that for each $n\in\Nat$, $\partial D(z_n+r_n\ba_n,\bom_n,r_n)\cap\cM$ contains at least two elements. Let $z_{n1}$ and $z_{n2}$ denote two such elements. Set $\bu_n\coloneqq (z_{n1}-z_{n2})/|z_{n1}-z_{n2}|$. The sequence of $\bu_n$ is bounded and, hence, has a subsequence that converges to, say, $\bu$. As $z_{n1},z_{n2}\in\cM$ and $z_{n1},z_{n2}\rightarrow z$ as $n\rightarrow\infty$, it follows that
\beqn\label{uTM}
\bu\in T_z\cM.
\eeqn
Moreover, since
\beqn
\bu_n\lrcorner\bom_n=\bzero,
\eeqn
taking the limit along the subsequence yields $\bu\lrcorner\bom=\bzero$, which is equivalent to $\bu\in[\bom]^\perp$. Putting this together with \eqref{uTM} we see that $\vol(z)\cdot\bom=0$, which contradicts the fact that $\Psi(z,\ba,\bom,r)\not\in \cD_\text{tan}$. Thus, $\hat\cA_{\partial D1}$ is an open subset of $\cM\times\cW\times\Real^+$, and it follows that $\hat\cA_{\partial D1}$ is a manifold of the same dimension.

 Our next goal is to show that $\Psi$ is an embedding. In particular, we need to confirm that it is an injective, continuous map with continuous inverse. Since no element of $\hat\cA_{\partial D1}$ is sent to a disk that is tangent to $\cM$, we see from \eqref{PsiJacobian} that the Jacobian of $\Psi$ is nonzero on $\hat\cA_{\partial D1}$ and, so, the gradient of $\Psi$ is injective on $\hat\cA_{\partial D1}$. This implies that $\Psi|_{\hat\cA_{\partial D1}}$, the restriction of $\Psi$ to $\hat\cA_{\partial D1}$, is an immersion. Moreover, one can readily see that $\Psi$ is injective on $\hat\cA_{\partial D1}$. Next, we show that the inverse of $\Psi|_{\hat\cA_{\partial D1}}:\hat\cA_{\partial D1}\rightarrow \cD_{\partial D1}\setminus(\cD_{\partial\cM}\cup\cD_{\rm tan})$ is continuous. Towards this end, consider a sequence $(p_n,\bom_n,r_n)$ in $\cD_{\partial D1}\setminus(\cD_{\partial\cM}\cup\cD_{\rm tan})$ that converges to $(p,\bom,r)$ that is also in this set. Find $(z_n,\bu_n,\bom_n,r_n)\in\hat\cA_{\partial D1}$ and $(z,\bu,\bom,r)\in\hat\cA_{\partial D1}$ such that
\[
\Psi(z_n,\bu_n,\bom_n,r_n)=(p_n,\bom_n,r_n)\quad\text{and}\quad \Psi(z,\bu,\bom,r)=(p,\bom,r).
\]
It suffices to show that $(z_n,\bu_n,\bom_n,r_n)\rightarrow (z,\bu,\bom,r)$. For the sake of contradiction, suppose that $\bu_n\not\rightarrow \bu$. Since $|\bu_n|=1$ for all $n$, it follows that there is a subsequence $\bu_{n_i}$ that converges to some $\hat\bu\not=\bu$ that must satisfy $\hat\bu\lrcorner\bom=\bzero$. As $p_n=z_n+r_n\bu_n$, taking the limit along this subsequence shows that $z_{n_i}$ converges to some $\hat z\in\overline\cM$ such that $p=\hat z+r\hat\bu$. From this we know that $\Psi(\hat z,\hat \bu,\bom,r)=\Psi(z,\bu,\bom,r)\in \cD_{\partial D1}\setminus(\cD_{\partial\cM}\cup\cD_{\rm tan})$. It follows that $\hat z\not\in\partial\cM$ and, thus, $(\hat z,\hat \bu,\bom,r)\in\hat\cA_{\partial D1}$. Moreover, as $\Psi$ is injective on $\hat\cA_{\partial D1}$, we must have $\hat \bu=\bu$ and $\hat z=z$, which yields the contradiction. We can thus conclude that the inverse of $\Psi|_{\hat\cA_{\partial D1}}$ is continuous and, so, $\Psi|_{\hat\cA_{\partial D1}}$ is an embedding. It follows that $\Psi(\hat\cA_{\partial D1})=\cD_{\partial D1}\setminus(\cD_{\partial\cM}\cup\cD_{\rm tan})$ is an embedded submanifold of $\cD$.

Now consider $(p,\bom,r)\in \cD_{\partial D1}\setminus(\cD_{\partial\cM}\cup\cD_{\rm tan})$ and find the unique $(z,\ba,\bom,r)\in\hat\cA_{\partial D1}$ such that $\Psi(z,\ba,\bom,r)=(p,\bom,r)$ so that $\ba=(p-z)/r$. Set
\beqn
\cT\coloneqq T_{(p,\bom,r)}(\cD_{\partial\cD1}\setminus(\cD_{\partial M1}\cup\cD_\text{tan})).
\eeqn
We now find a collection of vectors spanning this space.

If $\hat z:(-\ve,\ve)\rightarrow\cM$ is any smooth curve on $\cM$ such that $\hat z(0)=z$, then differentiating the curve
\beqn
s\mapsto \Psi(\hat z(s),\ba,\bom,r)=(\hat z(s)+r\ba,\bom,r)
\eeqn
and evaluating it at $s=0$ shows that $(\bt,\bzero,0)\in\cT$ for any $\bt\in T_z\cM$.

Next, for any $\bc\in \cU([\bom]^\perp)$ such that $\bc\cdot\ba=0$, define $\hat \ba:(-\ve,\ve)\rightarrow \cU([\bom]^\perp)$ by $\hat\ba(s)=\cos(s)\ba+\sin(s)\bc$. Notice that $(\hat\ba(s),\bom)\in\cW$ for any $s$. Thus, we can consider the curve
\beqn
s\mapsto \Psi(z,\hat \ba(s),\bom,r)=(z+r\hat\ba(s),\bom,r)
\eeqn
and by evaluating its derivative at $s=0$ we obtain that $(\bc,\bzero,0)\in\cT$. As $\cT$ is a linear space, this fact also holds when $\bc$ is not a unit vector.

For any $\bbeta\in T_{\bom} \hat\Lambda_u^k(\Real^n)$ with $\ba\lrcorner \bbeta=\bzero$, one can find a curve $\hat\bom:(-\ve,\ve)\rightarrow \Hat\Lambda_u^k(\Real^n)$ such that $\hat\bom'(0)=\bbeta$ and $(\ba,\hat\bom(s))\in\cW$ for all $s\in(-\ve,\ve)$. Thus, evaluating the derivative of the curve
\beqn
s\mapsto\Psi(z,\ba,\hat\bom(s),r)=(z+r\ba,\hat\bom(s),r)
\eeqn
at $s=0$ shows that $(\bzero,\bbeta,0)\in\cT$.

Now consider a unit vector $\bb\in[\bom]$ and functions $\hat\ba,\hat\bb:(-\ve,\ve)\rightarrow\Real^n$ defined by
\beqn
\hat\ba(s)=\cos(s)\ba+\sin(s)\bb\quad\text{and}\quad \hat\bb(s)=-\sin(s)\ba+\cos(s)\bb.
\eeqn
Recalling \eqref{replacewedgeID}, once can readily check that $(\hat\ba(s),\hat\bb(s)\wedge(\bb\lrcorner \bom))\in\cW$. Thus, we can consider the curve
\beqn
s\mapsto\Psi(z,\hat\ba(s),\hat\bb(s)\wedge(\bb\lrcorner \bom),r)=(z+r\hat\ba(s),\hat\bb(s)\wedge(\bb\lrcorner \bom),r)
\eeqn
and, upon evaluating its derivative at $s=0$, we see that $(\bb,-\ba\wedge(\bb\lrcorner\bom),0)\in\cT$. By linearity, this is also true for $\bb$ that are not unit vectors.

To show that $\cD_{\partial \cM1}\setminus(\cD_{\partial D}\cup\cD_\text{\rm tan})$ is an embedded submanifold of $\cD$ a similar argument can be used. The main difference is that one uses the function 
\beqn\label{xilem3.3}
\Xi(z,\ba,\bomega,\xi,r)\coloneqq (z+\xi \ba,\bomega,r),\quad (z,\ba,\bomega,\xi,r)\in \partial\cM\times\cW\times\{(\xi,r)\in\Real^+_0\times\Real^+\ :\ \xi\leq r\}.
\eeqn
The range of this function is $\cD_{\partial\cM}$ and it can be shown that if $\hat\cA_{\partial \cM1}\coloneqq \Xi^{-1}(\cD_{\partial \cM1}\setminus(\cD_{\partial D}\cup\cD_{\rm tan}))$, then $\Xi$ restricted to $\hat\cA_{\partial \cM1}$ is an embedding. Thus, $\cD_{\partial \cM1}\setminus(\cD_{\partial D}\cup\cD_\text{\rm tan})$ is an embedded submanifold of $\cD$. As the details are similar to the previous argument, they will be skipped.
\end{proof}

It turns out that $\cD_\text{odd}$ is locally a set of finite perimeter and its essential boundary can be described in terms of a few of the sets of disks introduced at the beginning of this section. To state this result it is useful to introduce some notation. If $A$ and $B$ are subsets of $\cD$, we write
\beqn\label{subsetsim}
A\subsetsim B\quad\text{if}\quad \cH^{d-1}(A\setminus B)=0
\eeqn
and
\beqn\label{cong}
A\cong B\quad \text{if}\quad A\subsetsim B\ \text{and}\ B\subsetsim A.
\eeqn

\begin{proposition}\label{oddFP}
The set $\D_\text{\rm odd}$ is locally a set of finite perimeter.  Moreover, the essential boundary $\partial^*\D_\text{\rm odd}$\footnote{For the definition of sets of finite perimeter and essential boundary see, for example, Ambrosio, Fusco, and Pallara \cite{AFP}.} of this set satisfies $\partial^*\D_\text{\rm odd}\cong(\D_{\partial\cM 1}\cup\cD_{\partial D1})\setminus \cD_\text{\rm tan}$.
\end{proposition}

\begin{proof}
It follows from Items~\ref{eo} and \ref{Dsep} of Lemma~\ref{lemmeasure2} that $\partial^*\cD_\text{odd}\subseteq \cD_{\partial\cM}\cup\cD_\text{tan}\cup\cD_{\partial D}$.  Thus, from Items~\ref{paM}--\ref{paD} of Lemma~\ref{lemmeasure}, whenever $\cE\subseteq \cD$ is a bounded, open set we have
\beqn
\cH^{d-1}(\partial^*\cD_\text{odd}\cap\cE)\leq \cH^{d-1}((\cD_{\partial\cM}\cup\cD_\text{tan}\cup\cD_{\partial D})\cap\cE)<\infty.
\eeqn
It then follows from a result of Federer \cite{Fed}[4.5.11] that $\cD_\text{odd}$ has finite perimeter in $\cE$.  A result of Ambriosio, Fusco, and Pallara \cite{AFP}[Theorem 3.61] then shows that $\cD_\text{odd}$ has density either $0$, $1/2$, or $1$ at $\cH^{d-1}$-a.e.~point of $\cE$.  Moreover, $\partial^*\cD_\text{odd}\cap\cE$ consists of those points with density $1/2$ up to a set of $\cH^{d-1}$-measure zero.

Next we show that at any point of $\cD_\text{tan}\setminus(\D_{\partial\M}\cup\D_{\partial D})$, the density of $\cD_\text{odd}$ is either 0 or 1.  Fix $(p,\bomega,r)\in\cD_\text{tan}\setminus(\D_{\partial\M}\cup\D_{\partial D})$ and find a neighborhood $\cN$ of it in $\cD$ such that $\cN$ is disjoint from $\D_{\partial\M}\cup\D_{\partial D}$.  This is possible as $\D_{\partial\M}$ and $\D_{\partial D}$ are closed in $\D$.  From Item~\ref{tan} of Lemma~\ref{lemmeasure}, $\cH^{d}$-a.e.~disk in $\cN$ is not tangent to $\M$ and, thus, is transversal to it.  Moreover, every disk in $\cN$ is homotopic to every other disk in this neighborhood.  Thus, from a result in differential topology, see Guillemin and Pollack \cite{GP74}, $\cH^{d}$-a.e.~disk in $\cN$ has the same number of intersections with $\M$ modulo 2.  It follows that the density of $\D_\text{odd}$ at $(p,\bomega,r)$ is either 0 or 1. Thus, $\partial^*\cD_\text{odd}\subseteq (\cD_{\partial\cM}\cup\cD_{\partial D})\setminus\cD_{\rm tan}$.

Since $\cH^{d-1}(\cD_{\partial D 2}\cup\cD_{\partial\cM 2})=0$, by Items~\ref{paD2} and \ref{paM2} of Lemma~\ref{lemmeasure2}, to complete the proof, it suffices to show that $(\D_{\partial D 1}\cup\cD_{\partial \M1})\setminus \cD_{\rm tan}\subsetsim\partial^*\cD_\text{odd}$.  First we show $\D_{\partial D 1}\setminus(\cD_{\rm tan}\cup\cD_{\partial\cM})\subsetsim\partial^*\cD_\text{odd}$.  From Lemma~\ref{lempDemb} we know that $\D_{\partial D 1}\setminus(\cD_{\rm tan}\cup\cD_{\partial\cM})$ is a $(d-1)$-dimensional embedded submanifold of $\cD$. It follows that small neighborhoods of a point in $\D_{\partial D 1}\setminus(\cD_{\rm tan}\cup\cD_{\partial\cM})$ are divided in half by this submanifold. Moreover, from the definition of the sets involved, one of these halves is $\cD_\text{odd}$ while the other is $\cD_\text{even}$. It follows that points of $\D_{\partial D 1}\setminus(\cD_{\rm tan}\cup\cD_{\partial\cM})$ cannot have density $0$ or $1$ relative to $\cD_\text{odd}$. Thus, it must be true that $\D_{\partial D 1}\setminus(\cD_{\rm tan}\cup\cD_{\partial\cM})\subsetsim\partial^*\cD_\text{odd}$.  A similar argument shows that $\D_{\partial\cM 1}\setminus(\cD_{\rm tan}\cup \cD_{\partial D})\subsetsim\partial^*\cD_\text{odd}$.
\end{proof}

Since $\cD_\text{odd}$ is locally of finite perimeter, we know that it has an exterior unit-normal at $\cH^{d-1}$-a.e.~point of its essential boundary.  The next result characterizes this normal on a part of $\partial^*\cD_\text{odd}$.

\begin{proposition}\label{ExtNormal}
For $\cH^{d-1}$-a.e.~$(p,\bomega,r)\in\partial^*\cD_\text{\rm odd}$ such that $D(p,\bomega,r)\cap\partial\cM$ is empty, there is a unique $z\in\partial D(p,\bomega,r)\cap\cM$ and $(p,\bom,r)\not\in\cD_\text{\rm tan}$.  For such $(p,\bomega,r)$ the exterior unit-normal $\bM(p,\bomega,r)\in \Real^n\times \Lambda^k(\Real^n)\times\Real$ to $\cD_\text{\rm odd}$ is given by
\beqn
\bM(p,\bomega,r)\coloneqq \begin{cases}
\phantom{-}\frac{\bm(p,\bomega,r)}{|\bm(p,\bomega,r)|} & \text{if } \big(\cH^0(D(p,\bomega,r)\cap \cM)\ \text{is odd and $\vol(z)\cdot\bomega>0\big)$}\\
 &\quad \text{or}\ \big((\cH^0(D(p,\bomega,r)\cap \cM)\ \text{is even and $\vol(z)\cdot\bomega<0\big)$},\\
-\frac{\bm(p,\bomega,r)}{|\bm(p,\bomega,r)|} & \text{if } \big(\cH^0(D(p,\bomega,r)\cap \cM)\ \text{is even and $\vol(z)\cdot\bomega>0\big)$}\\
 &\quad \text{or}\ \big((\cH^0(D(p,\bomega,r)\cap \cM)\ \text{is odd and $\vol(z)\cdot\bomega<0\big)$},
\end{cases}
\eeqn
where
\beqn\label{bm}
\bm(p,\bomega,r)\coloneqq ([(z-p)\wedge\bomega]\llcorner \vol(z), (p-z)\wedge[\bP_{\bom}((p-z)\lrcorner \vol(z))],r\vol(z)\cdot\bomega).
\eeqn
\end{proposition}

\begin{proof}
By Proposition \ref{oddFP}, for $\cH^{d-1}$-a.e.~$(p,\bomega,r)\in\partial^*\cD_\text{odd}$ such that $D(p,\bomega,r)\cap\partial\cM$ is empty, there is a unique $z\in\partial D(p,\bomega,r)\cap\cM$, and $(p,\bomega,r)\not\in\cD_\text{tan}$, so the disk $D(p,\bomega,r)$ is transversal to $\cM$. Fix such a $(p,\bom,r)$. From Lemma~\ref{lempDemb} we know that $\cD_{\partial D1}\setminus (\cD_{\partial \cM}\cup \cD_\text{tan})$ is an embedded manifold. Thus, to show that $\bM(p,\bom,r)$ is an external unit-normal vector, it suffices to establish that (i) $\bm(p,\bom,r)$ is orthogonal to the manifold $\cD_{\partial D1}\setminus (\cD_{\partial \cM}\cup \cD_\text{tan})$ at $(p,\bom,r)$, (ii) $\bm(p,\bom,r)$ is in the tangent space $T_{(p,\bom,r)}\cD$, and (iii) $\bM(p,\bom,r)$ points towards to exterior of $\cD_\text{odd}$.

First, we show that $\bm(p,\bom,r)$ is orthogonal to $T_{(p,\bom,r)}(\cD_{\partial D1}\setminus(\cD_{\partial \cM}\cup\cD_\text{\rm tan}))$. This is accomplished by showing that $\bm(p,\bom,r)$ is orthogonal the collection of vectors given in Lemma~\ref{lempDemb} that span this tangent space. Notice that $r=|p-z|$ and set $\ba\coloneqq (p-z)/r$.
Using these relations, \eqref{bm} can be written as
\beqn
\bm(p,\bom,r)=\big((-r\ba\wedge\bomega)\llcorner \vol(z), r\ba\wedge\bP_{\bom}(r\ba\lrcorner \vol(z)),r\vol(z)\cdot\bomega\big).
\eeqn
Looking at the vectors that span $T_{(p,\bom,r)}(\cD_{\partial D1}\setminus(\cD_{\partial \cM}\cup\cD_\text{\rm tan}))$ given in Lemma~\ref{lempDemb}, we begin by considering $\bt\in T_z\cM$ and notice that by the definition of the right interior-product \eqref{movingdotR} we have 
 \[
  (\tt,0,0)\cdot \bm(p,\bomega,r)=\tt\cdot\big((-r\ba\wedge \bom)\llcorner \vol(z)\big)=(\tt\wedge\vol(z))\cdot(-r\ba\wedge\bom)=0,
 \]
 since $\tt\in[\vol(z)]$. Next, consider any $\bc\in[\bom]^\perp$ such that $\bc\cdot\ba=0$. Similar to as before,   
 \[
 (\bc,0,0)\cdot \bm(p,\bomega,r)=(\bc\wedge\vol(z))\cdot(-r\ba\wedge\bom).
 \]
Find an orthonormal basis $\{\bw_1,\bw_2,\dots,\bw_k \}$ for $[\bom]$ and an orthonormal basis $\{\bt_1,\bt_2,\dots,\bt_k \}$ for $T_z\cM$ so that $\bom=\bw_1\wedge\dots\wedge\bw_k$ and $\vol(z)=\bt_1\wedge\dots\wedge\bt_k$. It follows that 
\beqn\label{kdot}
(\bc\wedge\vol(z))\cdot(\aa\wedge\bom)=\text{det}\begin{pmatrix}
\bc\cdot\aa&\bc\cdot\bw_1&\cdots&\bc\cdot\bw_k\\
\tt_1\cdot\aa&\tt_1\cdot\bw_1&\cdots&\tt_1\cdot\bw_k\\
\vdots&\vdots&\ddots&\vdots\\
\tt_k\cdot\aa&\tt_k\cdot\bw_1&\cdots&\tt_k\cdot\bw_k\\
\end{pmatrix},
\eeqn
which is zero since the first row is entirely made out of zeros due to the properties of $\bc$. Next, consider any $\bbeta\in T_{\bom} \hat\Lambda_u(\Real^n)$ such that $\ba\lrcorner\bbeta=\bzero$. It follows that
\[
(0,\bbeta,0)\cdot\bm(p,\bomega,r)=\bbeta\cdot \big(r\aa\wedge\big(\bP_{\bom}(r\ba\lrcorner \vol(z))\big)\big)=(r\aa\lrcorner\bbeta)\cdot \big(\bP_{\bom}(r\ba\lrcorner \vol(z))\big)=0.
\]
Next, fix $\bb\in[\bom]$ and set $\bxi\coloneqq \bb\lrcorner\bom$. Employing Proposition~\ref{dotreduction}, the definition of the left and right interior-products, and the fact that $\bP_{\bom}$ is symmetric, we can calculate
\begin{align*}
(r\bb,-\ba\wedge\bxi,0)\cdot \bm(p,\bomega,r)&=-r\bb\cdot\big((r\aa\wedge\bomega)\llcorner \vol\big)-(\aa\wedge\bxi) \cdot\big(r\aa\wedge\big(\bP_{\bom}(r\aa\lrcorner \vol)\big)\big)\\
&=-r^2\big(\bb\cdot\big((\aa\wedge\bomega)\llcorner \vol(z)\big)+\bxi\cdot \bP_{\bom}(\aa\lrcorner \vol(z))\big)\\
&=-r^2\big((\bb\wedge\vol(z))\cdot(\aa\wedge\bomega)+(\aa\wedge\bxi)\cdot\vol(z)\big).
\end{align*}
Now we use the fact that
\beqn
\ba\wedge\bom=\ba\wedge\bb\wedge(\bb\lrcorner\bom)=-\bb\wedge\ba\wedge\bxi
\eeqn
and Proposition~\ref{dotreduction} again to continue the calculation to find that
\begin{align*}
(r\bb,-\ba\wedge(\bb\lrcorner\bom),0)\cdot \bm(p,\bomega,r)&=-r^2\big(-(\bb\wedge\vol(z))\cdot(\bb\wedge\aa\wedge\bxi)+(\aa\wedge\bxi)\cdot\vol(z)\big)\\
&=-r^2\big(-\vol(z)\cdot(\aa\wedge\bxi)+(\aa\wedge\bxi)\cdot\vol(z)\big)=0.
\end{align*}
Lastly, from the definition of the right interior-product and Proposition~\ref{dotreduction} again,
\begin{multline}
(\aa,0,1)\cdot \bm(p,\bomega,r)=-\aa\cdot\big((r\aa\wedge\bomega)\llcorner \vol(z)\big)+r\vol(z)\cdot\bomega\\
=-(\aa\wedge\vol(z))\cdot(r\aa\wedge\bom)+r\vol(z)\cdot\bomega=0.
\end{multline}

Next, we show that $\bm(p,\bom,r)$ is in the tangent space $T_{(p,\bom,r)}\cD$. Since
\beqn
T_{(p,\bom,r)}\cD=\Real^n\times T_{\bom}\hat\Lambda_u^k(\Real^n)\times\Real,
\eeqn
it suffices to show that the second component of $\bm(p,\bom,r)$ is in $T_{\bom}\hat\Lambda_u^k(\Real^n)$. As we know $\bP_{\bom}(r\ba\lrcorner \vol(z))\in \Lambda^{k-1}([\bom])$, since $\bP_{\bom}$ is the projection onto $[\bom]$ and $\ba\in[\bom]^\perp$, this follows from Proposition~\ref{tusimpmv}.

We now show $\bM(p,\bom,r)$ points towards the exterior of $\cD_\text{odd}$. To begin, consider a smooth curve $\gamma$ in $\cD$ defined on an interval of $\Real$ containing zero such that $\gamma(0)=(p,\bom,r)$ and $\gamma'(0)=\bM(p,\bom,r)$. Suppose that $\cH^0(D(p,\bom,r)\cap\cM)$ is odd and $\vol(z)\cdot\bom>0$. The second of these conditions implies that the third component of $\bM(p,\bom,r)$ is positive. This means that the radius of the disk $\gamma(t)$ is an increasing function of $t$ for small $t$. Moreover, as there is a unique point $z\in \partial D(p,\bom,r)\cap\cM$, it follows that for small negative $t$, we have $\gamma(t)\in \cD_\text{odd}$ while for small positive $t$, $\gamma(t)\in \cD_\text{even}$. This shows that $\bM(p,\bom,r)$ is an exterior normal to $\cD_\text{odd}$ in this case. Now suppose that $\cH^0(D(p,\bom,r)\cap\cM)$ is even and $\vol(z)\cdot\bom<0$. The third component of $\bM(p,\bom,r)$ is now negative. Using similar reasoning as before, it follows that for small negative $t$, we have $\gamma(t)\in \cD_\text{odd}$ while for small positive $t$, $\gamma(t)\in \cD_\text{even}$. Hence, once again, this shows that $\bM(p,\bom,r)$ is an exterior normal to $\cD_\text{odd}$. The cases associated with $\bM(p,\bom,r)$ pointing in the opposite direction of $\bm(p,\bom,r)$ are argued in a similar way.
\end{proof}

The next three results are related and together show that the integral of $r^{-1-\sigma}$ over a certain collection of disks all of whose radius are less than $\eta$ goes to zero as $\eta$ goes to zero. This fact is useful in computing the first variation of the $\sigma$-measure in the next section.

\begin{lemma}\label{lemDr}
Given $z\in\R^n$, $\bmu\in\hat{\L}^k_u(\R^n)$, $K>0$, and $\a\in(0,1)$, define the region 
\beqn\label{bowtie}
\cR\coloneqq \{z+\bt+\bh : \tt\in[\bmu], \bh\in[\bmu]^{\perp},\, \text{and}\ |\bh|< K|\tt|^{1+\a}\}.
\eeqn
If 
\beqn\label{bowdiscs}
\cD_\cR\coloneqq \{(\aa,\bom,r)\in\cW\times (0,\infty):\bom\cdot \bmu>0\mbox{ and }D(z+r\aa,\bom,r)\cap \cR\neq \emptyset\},
\eeqn
then it follows that 
\beqn\label{cDRinc}
\cD_\cR\subseteq\Big\{(\aa,\bom,r)\in\cW\times(0,\infty):r\ge \frac{|\bom\cdot\bmu|^{1/\a}}{2K^{1/\a}}\Big\}.
\eeqn
\end{lemma}

\begin{proof}
Let $(\aa,\bom,r)\in\cD_\cR$. Notice that $|\bom\cdot \bmu|\neq 1$, as otherwise $D(z+r\aa,\bom,r)\cap\cR=\emptyset$. The main idea of this proof is to first consider the $(n-k)$-dimensional plane $z+[\bom]^\perp$ and show that the closest point to $z$ in the intersection of this plane with $\cR$ has distance at least $\frac{|\bom\cdot\bmu|^{1/\a}}{K^{1/\a}}$. Thus, for the disk $D(z+r\ba,\bom,r)$ to intersect $\cR$ would require $2r$ to be larger than this distance.

To minimize the distance between $z$ and the intersection of the plane $z+[\bom]^\perp$ and $\cR$, it suffices to minimize the distance to the intersection $(z+[\bom]^\perp)\cap \partial\cR$ not including $z$. Thus, we consider the minimization problem
\beqn\label{minproblem}
\inf|p-z|^2, \qquad p \in (z+[\bom]^\perp )\cap \pa\cR\setminus\{z\}. 
\eeqn
To analyze this problem, we employ the method of Lagrange multipliers. The constraints for $p$ can be formulated as 
\beqn\label{constraints}
\bom\llcorner(p-z)=\bzero\qquad\mbox{ and }\qquad |\bP_{\bmu}^\perp(p-z)|^2=K^2|\bP_{\bmu}(p-z)|^{2+2\a},
\eeqn
the first of these being associated with the plane $z+[\bom]^\perp$ and the second with $\partial\cR$. It follows that there are Lagrange multipliers $\l_1\in\R$ and $\blambda_2\in\Lambda^{k-1}([\bom])$ such that
\[\label{minplane}
2(p-z)=2\l_1\big((1+\a)K^2|\bP_{\bmu}(p-z)|^{2\a}\bP_{\bmu}(p-z)-\bP^\perp_{\bmu}(p-z)\big)+\blambda_2\lrcorner\bom.
\]
Using the identity $\bP_{\bmu}^\perp(p-z)=(p-z)-\bP_{\bmu}(p-z)$ and rearranging the previous equation yields
\beqn\label{newminplane}
2(1+\l_1)(p-z)=2\l_1\big((1+\a)K^2|\bP_{\bmu}(p-z)|^{2\a}+1\big)\bP_{\bmu}(p-z)+\blambda_2\lrcorner\bom.
\eeqn
The goal is to use this equality to obtain a lower bound for $|p-z|$. To do so, we need to handle three degenerate cases before getting to the fourth, main case.

\medskip
\noindent Case 1: Suppose $\bP_{\bmu}(p-z)=\bzero$. Then, by \eqref{constraints}$_2$, $\bP^{\perp}_{\bmu}(p-z)=\bzero$ and, so, $p=z$, which is not allowed in \eqref{minproblem}.

Thus, for the rest of the cases, we can use the unit vector
\beqn
\uu\coloneqq \frac{\bP_{\bmu}(p-z)}{|\bP_{\bmu}(p-z)|}.
\eeqn

\medskip
\noindent Case 2: Suppose $\bP_{\bmu}(p-z)\neq \bzero$ and $\l_1=-1$. Using 
\[
\overline{\l}_1\coloneqq 2\big[(1+\a)K^2|\bP_{\bmu}(p-z)|^{2\a}+1\big]|\bP_{\bmu}(p-z)|,
\]
\eqref{newminplane} can be written as
\[
\bzero=\overline{\l}_1\uu+\blambda_2\lrcorner\bom.
\]
As $\blambda_2\lrcorner\bom\in[\bom]$, this implies $\uu\in[\bom]$. We know $(p-z)\in[\bom]^\perp$, so
\[
0=\uu\cdot (p-z)=\bP_{\bmu}\uu\cdot (p-z)=\uu\cdot \bP_{\bmu}(p-z)=|\bP_{\bmu}(p-z)|,
\]
which is a contradiction. 

\medskip
\noindent Case 3: Suppose $\bP_{\bmu}(p-z)\neq \bzero$, $\l_1\neq-1$, and $\bP_{\bmu}\bP_{\bom}^\perp\bu=\bzero$. Upon defining
\beqn\label{betaudefcase3}
\beta_1\coloneqq \frac{\l_1[(1+\a)K^2|\bP_{\bmu}(p-z)|^{2\a}+1]|\bP_{\bmu}(p-z)|}{1+\l_1} \quad \text{and}\quad \bbbeta_2\coloneqq \frac{\blambda_2}{2(1+\l_1)},
\eeqn
\eqref{newminplane} becomes
\beqn\label{newminplanecase3}
p-z=\b_1\uu+\bbbeta_2\lrcorner\bom.
\eeqn
Utilizing \eqref{constraints}$_1$, we find that
\[
\bzero=\bom\llcorner( \b_1\uu+\bbbeta_2\lrcorner\bom)=\beta_1\bom\llcorner\uu+\bom\llcorner(\bbbeta_2\lrcorner \bom).
\]
Since $[\blambda_2]\subseteq[\bom]$, it is true that $[\bbbeta_2]\subseteq[\bom]$, so it follows that $\bbbeta_2=\bom\llcorner(\bbbeta_2\lrcorner \bom)$. Therefore, the previous equation implies that $\bbbeta_2=-\b_1\bom\llcorner \uu$.
Hence, going back to \eqref{newminplanecase3}, using \eqref{projip} yields
\[
p-z=\b_1(\uu-(\bom\llcorner\uu)\lrcorner\bom)=\b_1(\uu-\bP_{\bom}\uu)=\b_1 \bP^\perp_{\bom}\uu.
\]
Plugging this into \eqref{constraints}$_2$ gives
\[
|\bP_{\bmu}^\perp \bP^\perp_{\bom}\uu|^2=\b_1^{2\a}K^2|\bP_{\bmu}\bP^\perp_{\bom}\uu|^{2+2\a}\,.
\]
Since the right-hand side is 0 in this case, we conclude $\bP^\perp_{\bom}\uu=\bzero$, so $\uu\in[\bom].$ As in Case 2, this gives a contradiction. 

\medskip
\noindent Case 4: We are now ready to handle the main case, in which $\bP_{\bmu}(p-z)\neq \bzero$, $\l_1\neq-1$, and $\bP_{\bmu}\bP_{\bom}^\perp\bu\neq \bzero$. Define $\beta_1$ and $\bbbeta_2$ as in \eqref{betaudefcase3}. As in Case 3,
\[
|\bP_{\bmu}^\perp \bP^\perp_{\bom}\uu|^2=\b_1^{2\a}K^2|\bP_{\bmu}\bP^\perp_{\bom}\uu|^{2+2\a}\quad\text{and}\quad p-z=\beta_1\bP_{\bom}^\perp\bu.
\]
It follows that
\[
\beta_1=\frac{|\bP_{\bmu}^\perp \bP^\perp_{\bom}\uu|^{1/\a}}{K^{1/\a}|\bP_{\bmu}\bP^\perp_{\bom}\uu|^{1+1/\a}}
\]
and, so,
\beqn\label{pzmag}
|p-z|=\frac{|\bP_{\bmu}^\perp \bP_{\bom}^\perp\uu|^{1/\alpha}|\bP_{\bom}^\perp\uu|}{K^{1/\a}|\bP_{\bmu}\bP_{\bom}^\perp\bu|^{1+1/\a}}.
\eeqn

To help simplify the right-hand side of the previous equation, find an orthonormal basis $\{\bm_1,\bm_2,\dots,\bm_k\}$ for $[\bmu]$ and $\{\bw_1,\bw_2,\dots,\bw_k\}$ for $[\bom]$ such that
\beqn\label{anglebasis}
\bm_i\cdot \bw_j=0 \mbox{ if } i\neq j\qquad \mbox{ and }\qquad \bm_i\cdot \bw_i=\cos\theta_i,
\eeqn
where $0\le \theta_1\le \theta_2\le...\le \theta_k\le \pi/2$ are the principal angles between the subspaces $[\bmu]$ and $[\bom]$; see Jordan \cite{J1875}. Moreover, the order of these vectors can be chosen so that $\bmu=\bm_1\wedge\dots\wedge\bm_k$ and $\bom=\bw_1\wedge\dots\wedge\bw_k$ so that from \eqref{kvip},
\beqn\label{padotprod}
\bmu\cdot\bomega=\prod_{i=1}^k\cos\theta_i.
\eeqn
Since $\uu\in [\bmu]$, we can write $\uu=u_1\bm_1+...+u_k\bm_k$ for real numbers $u_1,...,u_k$. It follows that
\[
\bP_{\bom}^\perp\bu=\bu-\sum_{i=1}^k(\bw_i\otimes\bw_i)\bu=\bu-\sum_{i=1}^k u_i\cos\theta_i\bw_i=\sum_{i=1}^k(u_i(\bm_i-\cos\theta_i\bw_i))
\]
and, so, 
\beqn\label{magnitude1}
|\bP_{\bom}^\perp\bu|=\sqrt{\sum_{i=1}^k(u_i^2|\bm_i-\cos\theta_i\bw_i|^2)}=\sqrt{\sum_{i=1}^ku_i^2\sin^2\theta_i}.
\eeqn
We also have
\[
\bP_{\bmu} \bP_{\bom}^\perp\bu=\sum_{j=1}^k\Big[(\bm_j\otimes\bm_j)\sum_{i=1}^k(u_i(\bm_i-\cos\theta_i\bw_i))\Big]=\sum_{i=1}^{k}u_i\sin^2\theta_i\bm_i
\]
and, hence,
\beqn\label{magnitude2}
|\bP_{\bmu} \bP_{\bom}^\perp\bu|=\sqrt{\sum_{i=1}^{k}u_i^2\sin^4\theta_i}.
\eeqn
Moreover, this yields
\[
\bP_{\bmu}^\perp \bP_{\bom}^\perp\bu=\bP_{\bom}^\perp \bu-\bP_{\bmu} \bP_{\bom}^\perp\bu=\sum_{i=1}^ku_i(\cos^2\theta_i\bm_i-\cos\theta_i\bw_i)
\]
and 
\beqn\label{magnitude3}
|\bP_{\bmu}^\perp \bP_{\bom}^\perp\bu|=\sqrt{\sum_{i=1}^{k}u_i^2\cos^2\theta_i|\cos\theta_i\bm_i-\bw_i|^2}=\sqrt{\sum_{i=1}^ku_i^2\cos^2\theta_i\sin^2\theta_i}.
\eeqn
Inserting \eqref{magnitude1}--\eqref{magnitude3} into \eqref{pzmag} yields
\beqn
|p-z|=\left[\frac{\displaystyle\Big(\sum_{i=1}^ku_i^2\cos^2\theta_i\sin^2\theta_i \Big)\Big(\sum_{i=1}^ku_i^2\sin^2\theta_i\Big)^\a}{K^2\Big(\displaystyle\sum_{i=1}^ku_i^2\sin^4\theta_i\Big)^{\alpha+1}} \right]^{\frac{1}{2\alpha}} \ge \left[\frac{\displaystyle\sum_{i=1}^ku_i^2\cos^2\theta_i\sin^2\theta_i}{K^2\displaystyle\sum_{i=1}^ku_i^2\sin^2\theta_i}    \right]^{\frac{1}{2\alpha}}.
\eeqn
Now set 
\[
\xi_i\coloneqq\frac{u_i^2\sin^2\theta_i}{\sum_{i=1}^ku_i^2\sin^2\theta_i},
\]
and notice that $0\le \xi_i\le 1$ for all $i$ and $\sum_{i=1}^k\xi_i=1$. By the generalized arithmetic-geometric mean inequality and \eqref{padotprod}
\begin{multline}
|p-z|\ge K^{-1/\a}\left[\sum_{i=1}^k\xi_i\cos^2\theta_i\right]^{\frac{1}{2\a}}\ge K^{-1/\a}\left[\prod_{i=1}^k(\cos^2\theta_i)^{\xi_i}\right]^{\frac{1}{2\a}}\\
\ge K^{-1/\a}\left[\prod_{i=1}^k\cos^2\theta_i\right]^{\frac{1}{2\a}}= K^{-1/\a}\left[\prod_{i=1}^k\cos\theta_i\right]^{\frac{1}{\a}}=\frac{(\bmu\cdot\bom)^{1/\a}}{K^{1/\a}}.
\end{multline}

Since $D(z+r\aa,\bom,r)\cap\cR\not=\emptyset$, we must have $2r\geq |p-z|$, where $p$ achieves the infimum in \eqref{minproblem}. So,
\[
2r\geq |p-z|\ge \frac{(\bmu\cdot\bom)^{1/\a}}{K^{1/\a}}.
\]
As $(\ba,\bom,r)\in\cD_\cR$ was arbitrary, this establishes \eqref{cDRinc}.
\end{proof}

To state the next lemma, we need additional notation. Let $\cX$ be a $q$-dimensional subspace of $\Real^n$. Denote by
\beqn\label{Ukdef0}
\mathcal{U}^k_{0,\perp}(\cX)\coloneqq\{(\uu_1,...,\uu_k):\uu_i\in\cU(\cX) \text{ or } \uu_i=\mathbf{0}, \uu_i\cdot \uu_j=0, 1\leq i<j\leq k\}
\eeqn
the set of $k$ orthogonal vectors in $\cX$ that either have unit length or are zero. For $q\ge k$ denote by
\beqn\label{Ukdef}
\mathcal{U}^k_\perp(\cX)\coloneqq\{(\uu_1,...,\uu_k):\uu_i\in\cU(\cX), \uu_i\cdot \uu_j=0, 1\leq i<j\leq k\}
\eeqn
the set of $k$ orthonormal vectors that live in the subspace. Using these sets it is possible to parameterize $\hat\Lambda^k_u(\Real^n)$. Namely, given $\bmu\in\hat{\Lambda}_u^k(\R^n)$, 
define $\Phi_{\bmu}:\cU^k_\perp([\bmu])\times \cU^k_{0,\perp}([\bmu]^\perp)\times[0,\pi/2]^k\rightarrow \hat\Lambda^k(\R^n)$ by 
\beqn\label{fareadef}
\Phi_{\bmu}(\bm_1,..,\bm_k,\bu_1,...,\bu_k,\theta_1,...,\theta_k)=\bigwedge_{i=1}^k\big(\cos\theta_i\bm_i+\sin\theta_i\bu_i).
\eeqn
From the theory of angles between subspaces, see Jordan \cite{J1875}, it follows that $\Phi_{\bmu}$ is surjective.

\begin{lemma}\label{parsubspace} 
The gradient of the function $\Phi_{\bmu}$ defined in \eqref{fareadef} is uniformly bounded, meaning there is a $C>0$ such that
\beqn\label{Phibound}
\|\nabla\Phi_{\bmu}\|_{\infty}\coloneqq \sup\{|\nabla_\bA \Phi_{\bmu}| : \bA\in\cU^k_\perp([\bmu])\times \cU^k_{0,\perp}([\bmu]^\perp)\times[0,\pi/2]^k\}<C.
\eeqn
Moreover, $C$ is independent of $\bmu\in \hat\Lambda^k_u(\Real^n)$.
\end{lemma}

\begin{proof}To prove this, we will show all the directional derivatives of $\Phi_{\bmu}$ at any point are uniformly bounded. Towards that end, begin by fixing
\beqn
\bA=(\bm_1,\dots,\bm_k,\bu_1,\dots,\bu_k,\theta_1,\dots,\theta_k)\in\cU^k_\perp([\bmu])\times \cU^k_{0,\perp}([\bmu]^\perp)\times[0,\pi/2]^k.
\eeqn
Notice that $\dim\,[\bmu]^\perp=n-k$ may be strictly less than $k$. Thus, without loss of generality, we may assume that the first $\ell$ of the $\bu_i$ are nonzero, where $\ell$ is an integer between $1$ and $n-k$, with the rest being zero. Find an orthonormal basis of $\Real^n$ of the form $\{\bm_1,\dots\bm_k,\bu_1,\dots,\bu_\ell,\be_1,\dots\be_{n-k-\ell}\}$. Let $\bzero_k\in(\Real^n)^k$ be the list of $k$ zero vectors and $\bzero_\theta\in\Real^k$ be the $k$-dimensional zero vector that will be associated with the space of angles $[0,\pi/2]^k$. The tangent space $T_\bA(\cU^k_\perp([\bmu])\times \cU^k_{0,\perp}([\bmu]^\perp)\times[0,\pi/2]^k)$ consists of elements of the following form: 

\begin{enumerate}

\item $\{(\bzero_k,\bE_{ij},\bzero_\theta): 1\le i\le \ell, 1\le j\le n-k-\ell\},$ where $\bE_{ij}\coloneqq(\mathbf{0},...,\ee_j,...,\mathbf{0})$ is a list of $k$ vectors with $\ee_j$ in the $i^{\rm th}$ position. This set contains $k(n-k-\ell)$ vectors.

\item $\{(\bF_{ij},\bzero_k,\bzero_\theta): 1\le i<j\le k\}$, where $\bF_{ij}\coloneqq\frac{1}{\sqrt{2}}(\textbf{0},..,\bu_j,...,-\bu_i,...,\textbf{0})$ is a list of $k$ vectors with $\uu_j$ in the $i^{\rm th}$ position and $-\uu_i$ in the $j^{\rm th}$ position. This set contains $k(k-1)/2$ vectors.

\item $\{(\bzero_k,\bF_{ij},\bzero_\theta): 1\le i<j\le \ell\}$, where $\bF_{ij}\coloneqq\frac{1}{\sqrt{2}}(\textbf{0},..,\bu_j,...,-\bu_i,...,\textbf{0})$ is a list of $k$ vectors with $\uu_j$ in the $i^{\rm th}$ position and $-\uu_i$ in the $j^{\rm th}$ position. This set contains $\ell(\ell-1)/2$ vectors.

\item $\{(\bzero_k,\bzero,\bg_i): 1\le i\le k,\},$ where $\bg_{i}\coloneqq(0,\dots,1,\dots,0)$ is a list of $k$ numbers with $1$ in the $i^{\rm th}$ position. This set contains $k$ vectors.

\end{enumerate}

\noindent We now bound the gradient of $\Phi_{\bmu}$ at $\bA$, $\nabla_\bA \Phi_{\bmu}$, acting on each of the vectors listed above. To do so, we use the abbreviation $\bff_i\coloneqq \cos\theta_i\bm_i+\sin\theta_i\bu_i$.

\begin{enumerate}

\item It is true that $|\nabla_\bA \Phi_{\bmu}(\bzero_k,\bE_{ij},\bzero_\theta)|\leq 1$ since
\beqn
\nabla_\bA \Phi_{\bmu}(\bzero_k,\bE_{ij},\bzero_\theta)=\bff_1\wedge...\wedge \underset{i^{\rm th} \text{ slot}}{\sin\theta_i\be_j}\wedge...\wedge\bff_k.
\eeqn

\item It is true that $|\nabla_\bA \Phi_{\bmu}(\bF_{ij},\bzero_k,\bzero_\theta)|\leq 2$ since
\begin{align*}
\hspace{-.5in}\nabla_\bA \Phi_{\bmu}(\bF_{ij},\bzero_k,\bzero_\theta)&=\frac{1}{\sqrt{2}}(\bff_1\wedge...\wedge \underset{i^{\rm th} \text{ slot}}{\cos\theta_i\bm_j}\wedge...\wedge\bff_k)-\frac{1}{\sqrt{2}}(\bff_1\wedge...\wedge \underset{j^{\rm th} \text{ slot}}{\cos\theta_j\bm_i}\wedge...\wedge\bff_k).
\end{align*}

\item  It is true that $|\nabla_\bA \Phi_{\bmu}(\bzero_k,\bF_{ij},\bzero_\theta)|\leq 2$ since
\begin{align*}
\hspace{-.5in}\nabla_\bA \Phi_{\bmu}(\bzero_k,\bF_{ij},\bzero_\theta)&=\frac{1}{\sqrt{2}}(\bff_1\wedge...\wedge \underset{i^{\rm th} \text{ slot}}{\sin\theta_i\bu_j}\wedge...\wedge\bff_k)-\frac{1}{\sqrt{2}}(\bff_1\wedge...\wedge \underset{j^{\rm th} \text{ slot}}{\sin\theta_j\bu_i}\wedge...\wedge\bff_k).
\end{align*}

\item It is true that $|\nabla_\bA \Phi_{\bmu}(\bzero_k,\bzero_k,\bg_i)|\leq 1$ since
\beqn
\hspace{-.4in}\nabla_\bA \Phi_{\bmu}(\bzero_k,\bzero_k,\bg_i)=\bff_1\wedge...\wedge \underset{i^{\rm th} \text{ slot}}{(-\sin\theta_i\bm_i+\cos\theta_i\bu_i)}\wedge...\wedge\bff_k.
\eeqn

\end{enumerate}

\noindent As all of the directional derivatives of $\Phi_{\bmu}$ are bounded independently of $\bA$ and $\bmu$, it follows that $\nabla \Phi_{\bmu}$ is uniformly bounded and this bound is independent of $\bmu$.
\end{proof}

The pervious result shows that the Jacobian $J_{\Phi_{\bmu}}=\sqrt{\det(\nabla \Phi_{\bmu}^\trans\nabla \Phi_{\bmu})}$ is uniformly bounded. It follows that for any positive, integrable function $g$ defined on $\hat\Lambda^k_u(\Real^n)$, the area formula yields
\beqn\label{sukvecarea}
\int_{\hat\Lambda^k_u(\Real^n)} g(\bom)\, d\bom\leq \|J_{\Phi_{\bmu}}\|_{L^\infty} \int\limits_{\cU_\perp^k([\bmu])}\int\limits_{\cU_{0,\perp}^k([\bmu]^\perp)}\mathop{\int...\int}_{\substack{[0,\pi/2]^k}}g(\Phi_{\bmu}(\bm,\bu,\btheta))\,d\btheta d\bu d\bm,
\eeqn
where the notation $\bm=(\bm_1,\dots,\bm_k)\in\cU_\perp^k([\bmu])$, $\bu=(\bu_1,\dots,\bu_k)\in\cU_{0,\perp}^k([\bmu]^\perp)$, and $\btheta=(\theta_1,\dots,\theta_k)\in [0,\pi/2]^k$ has been used.

\begin{lemma}\label{0bowlem}
With $z\in \R^n$, $\bmu\in\hat{\Lambda}_u^k(\R^n)$, $K>0$, and $\a>\s$, define $\cR$ and $\cD_{\cR}$ as in Lemma \ref{lemDr}. If 
\beqn\label{cDcReta}
\cD_\cR(\eta)\coloneqq \{(\aa,\bom,r)\in\cD_\cR:r<\eta\},
\eeqn
then for sufficiently small $\eta>0$,
\[
\int_{\D_{\cR}(\eta)}r^{-1-\s} \, d\H^{w+1}(\ba,\bom,r)\le C\eta^{\a-\s}(-\ln[ K(2\eta)^\a])^{k-1},
\]
which converges to zero as $\eta\downarrow 0$ uniformly in $\bmu.$
\end{lemma}
\begin{proof}

The arguments in Mihaila and Seguin \cite{MS25}[Lemma~4] show that for an integrable function $h$, the coarea formula yields 
\begin{equation}\label{gcoarea}
\int_{\W}h(\ba,\bom)\, d\H^w(a,\bom)=2^{k/2}\int_{\hat{\Lambda}_u^k(\R^n)}\int_{\cU([\bom]^\perp)}h(\ba,\bom)\, d\H^{k(n-k)}(\bom)d\H^{n-1}(\ba).
\end{equation}
Now set 
\[
\hat{\Lambda}_u^k(\eta)\coloneqq \{\bom\in \hat{\Lambda}_u^k(\R^n):0<\bom\cdot \bmu\le K(2\eta)^\alpha\}
\]
and $\ve\coloneqq K(2\eta)^\a$. Then, using \eqref{gcoarea} and Lemma~\ref{lemDr}, followed by a calculation, shows that
\begin{align}
\int_{\D_{\cR}(\eta)}r^{-1-\s} \, d\H^w(\ba,\bom,r)&\le2^{k/2}\int_{  \hat{\Lambda}_u^k(\eta)}\int_{\cU([\bom]^\perp)}\int_{\frac{|\bom\cdot \bmu|^{1/\alpha}}{2K^{1/\a}}}^\eta r^{-1-\s}\,drd\ba d\bom \\
&=\frac{2^{k/2}\om_{n-k-1}}{\s}\int_{\hat{\Lambda}_u^k(\eta)}\Big(\frac{2^{\s}K^{\s/\alpha}}{|\bom\cdot\bmu|^{\s/\a}}-\eta^{-\s}\Big)\,d\bom,
\end{align}
where $\omega_{n-k-1}$ is the surface area of the unit ball in $\Real^{n-k}$. Next, using \eqref{sukvecarea}, utilizing the notation $\btheta=(\theta_1,\dots,\theta_k)\in[0,\pi/2]^k$, $\bm=(\bm_1,\dots,\bm_k)\in\cU_{0,\perp}^k([\bmu])$, and $\bu=(\bu_1,\dots,\bu_k)\in\cU_\perp^k([\bmu]^\perp)$, and the fact that
\beqn
|\Phi_{\bmu}(\bm,\bu,\btheta)\cdot\bmu|=\prod_{i=1}^k\cos\theta_i,
\eeqn
the calculation can be continued as
\begin{align}
\int_{\hat{\Lambda}_u^k(\eta)}\Big(\frac{2^{\s}K^{\s/\alpha}}{|\bom\cdot\bmu|^{\s/\a}}-\eta^{-\s}\Big)\,d\bom&\leq \|J_{\Phi_{\bmu}}\|_{L^\infty} \int\limits_{\cU_\perp^k([\bmu])}\int\limits_{\cU_{0,\perp}^k([\bmu]^\perp)}\mathop{\int...\int}_{\substack{\prod_{i=1}^k\cos\theta_i\le \ve\\\theta_i\in[0,\pi/2]}}\big(2^\s K^{\s/\a}\prod_{i=1}^k\cos^{-\s/\a}\theta_i-\eta^{-\s}\big)d\btheta d\bu d\bm\\
&=C_1 \mathop{\int...\int}_{\substack{\prod_{i=1}^k\cos\theta_i\le \ve\\\theta_i\in[0,\pi/2]} }\Big(2^\s K^{\s/\a}\prod_{i=1}^k\cos^{-\s/\a}\theta_i-\eta^{-\s}\Big)\,\d\btheta,
\end{align}
where
\begin{equation}\label{C1const}
C_1\coloneqq \|J_{\Phi_{\bmu}}\|_{L^\infty} \int\limits_{\cU_\perp^k([\bmu])}\int\limits_{\cU_{0,\perp}^k([\bmu]^\perp)}d\bu d\bm<\infty.
\end{equation}
After utilizing the inequality $1-\frac{2}{\pi}\theta_i\leq\cos\theta_i$ for $\theta_i\in[0,\pi/2]$, we apply two substitutions: first $x_i=1-\frac{2}{\pi}\theta_i$ and then $v_i=-\ln x_i$, to get 
\begin{align}
\mathop{\int...\int}_{\substack{\prod_{i=1}^k\cos\theta_i\le \ve\\ \theta_i\in[0,\pi/2]} }\Big(2^\s K^{\s/\a}\prod_{i=1}^k\cos^{-\s/\a}\theta_i-\eta^{-\s}\Big)\,\d\btheta&\leq\mathop{\int...\int}_{\substack{\prod_{i=1}^k(1-\frac{2}{\pi}\theta_i)\le \ve\\ \theta_i\in[0,\pi/2]} }\Big(2^\s K^{\s/\a}\prod_{i=1}^k\Big(1-\frac{2}{\pi}\theta_i\Big)^{-\s/\a}-\eta^{-\s}\Big)\,\d\btheta\\
&\hspace{-1in}= \mathop{\int...\int}\limits_{\substack{\prod_{i=1}^k x_i\le \ve\\ x_i\in[0,1]} }\left(2^\s K^{\s/\a}\prod_{i=1}^k x_i^{-\s/\a}-\eta^{-\s}\right)\Big(\frac{\pi}2\Big)^{k}\,d\bx\\
&\hspace{-1in}=\Big(\frac{\pi}2\Big)^{k} \mathop{\int...\int}\limits_{\substack{\sum_{i=1}^k v_i\ge -\ln \ve\\ v_i\in[0,\infty)} }\left[2^\s K^{\s/\a}\exp\Big(\sum_{i=1}^k\frac{\s}{\a}v_i\Big)-\eta^{-\s}\right]\exp\Big(\sum_{i=1}^k v_i\Big)\, d\bv.
\end{align}
Next, use the coarea formula where the domain is sliced using the level sets of the mapping $\bv\mapsto v_1+\dots+v_k$ and the fact that for any $t\in(0,\infty)$,
\beqn
\cH^{k-1}\Big(\big\{\bv\in\Real^k : \sum_{i=1}^kv_i=t,v_i\ge 0\big\}\Big)=\frac{\sqrt{k}t^{k-1}}{(k-1)!}
\eeqn
to find that
\begin{align}
\mathop{\int...\int}\limits_{\substack{\sum_{i=1}^k v_i\ge -\ln \ve\\ v_i\in[0,\infty)} }&\left[2^\s K^{\s/\a}\exp\Big(\sum_{i=1}^k\frac{\s}{\a}v_i\Big)-\eta^{-\s}\right]\exp\Big(\sum_{i=1}^k v_i\Big)\, d\bv \\
&\hspace{.5in}=\int_{-\ln(\ve)}^\infty k^{-1/2}\int_{\sum_{i=1}^k v_i=t}\left[2^\s K^{\s/\a}e^{\frac{\s}{\a}t}-\eta^{-\s}\right]e^{-t}\, \d\H^{k-1}(\bv)\, dt\\
&\label{twointegrals}\hspace{.5in}=\int_{-\ln(\ve)}^\infty \frac{t^{k-1}}{(k-1)!}\left[2^\s K^{\s/\a}e^{-(1-\frac{\s}\a) t}-\eta^{-\sigma}e^{-t}\right]\, dt.
\end{align}
If $k=1$, then 
\begin{align*}
\int_{-\ln(\ve)}^\infty \frac{t^{k-1}}{(k-1)!}\left[2^\s K^{\s/\a}e^{-(1-\frac{\s}\a) t}-\eta^{-\sigma}e^{-t}\right]\, dt&=\int_{-\ln(\ve)}^\infty\left[2^\s K^{\s/\a}e^{-(1-\frac{\s}\a) t}-\eta^{-\sigma}e^{-t}\right]\,dt\\
&=2^\a K\eta^{\a-\s}\frac{\s}{\a-\s}.
\end{align*}
For $k>1$, \eqref{twointegrals} is the difference of two integrals. We represent the integrals in terms of the upper incomplete gamma function, which has a known formula when $k\in\Nat$ and $k>1$:
\[
\Gamma(k,x)\coloneqq \int_x^\infty t^{k-1}e^{-t}\, dt=(k-1)! e^{-x}\sum_{i=1}^{k-1}\frac{x^i}{i!};
\]
see Jameson \cite{J2016}. For the first of the integrals in \eqref{twointegrals}, applying the substitution $u=(1-\frac{\s}\a) t$ we get
\begin{align*}
\int_{-\ln(\ve)}^\infty \frac{t^{k-1}}{(k-1)!}2^\s K^{\s/\a}e^{-(1-\frac{\s}\a) t}\, dt
&=\Big(1-\frac{\s}{\a}\Big)^{-k}\int_{-(1-\tfrac{\s}{\a})\ln \ve}^\infty \frac{u^{k-1}}{(k-1)!}2^\s K^{\s/\a}e^{-u}\, du\\
&=\Big(1-\frac{\s}{\a}\Big)^{-k}\frac{2^\s K^{\s/\a}}{(k-1)!}\Gamma\big(k, -(1-\tfrac{\s}{\a})\ln\ve\big )\\
&=\Big(1-\frac{\s}{\a}\Big)^{-k} 2^\s K^{\s/\a}\ve^{1-\tfrac{\s}{\a}}\sum_{i=1}^{k-1}\frac{(-(1-\tfrac{\s}{\a})\ln \ve)^i}{i!}.
\end{align*}
The second integral in \eqref{twointegrals} can be expressed as
\beqn
\int_{-\ln(\ve)}^\infty \frac{t^{k-1}}{(k-1)!}\eta^{-\sigma}e^{-t}\, dt=\frac{\eta^{-\sigma}}{(k-1)!}\Gamma(k,-\ln \ve)=\eta^{-\s}\ve\sum_{i=1}^{k-1}\frac{(-\ln \ve)^i}{i!}.
\eeqn
As $\ve=K(2\eta)^\alpha$, for small $\eta>0$, 
\begin{align}
\int_{-\ln(\ve)}^\infty \frac{t^{k-1}}{(k-1)!}\Big[2^\s K^{\s/\a}&e^{-(1-\frac{\s}\a) t}-\eta^{-\sigma}e^{-t}\Big]\, dt\\
&=\eta^{\a-\s}K2^{\a}\left(\Big(1-\frac{\s}{\a}\Big)^{-k} \sum_{i=1}^{k-1}\frac{(-(1-\tfrac{\s}{\a})\ln \ve)^i}{i!}-\sum_{i=1}^{k-1}\frac{(-\ln \ve)^i}{i!}\right)\\
&=\eta^{\a-\s}K2^{\a}\sum_{i=1}^{k-1}\frac{[(1-\tfrac{\s}{\a})^{i-k}-1](-\ln \ve)^i}{i!}\\
&\le \eta^{\a-\s}K2^{\a}(k-1)((1-\tfrac{\s}{\a})^{1-k}-1)(-\ln \ve)^{k-1}.
\end{align}
Thus, for any $k\in\{1,\dots,n-1\}$, there is a constant $C$ dependent only on $K, \a,\s,k,n$, and $C_1$ from \eqref{C1const}, such that 
\[
\int_{\D_{\cR}(\eta)}r^{-1-\s} \, d\H^w(\ba,\bom,r)\le C\eta^{\a-\s}(-\ln[ K(2\eta)^\a])^{k-1}.
\]
Notice that, because $\alpha>\sigma$, as $\eta\downarrow 0,$ the integral converges to 0. 
\end{proof}


\section{First variation of the fractional measure} \label{sec:4}
\setcounter{section}{4} \setcounter{equation}{0} 
In this section we compute the first variation of the fractional $k$-dimensional measure that was introduced by Mihaila and Seguin \cite{MS25}. As mentioned in the introduction, given an open, bounded $\Om\subseteq\Real^n$, the $k$-dimensional $\s$-measure of $\M$ relative to $\Om$ is defined by
\begin{equation}
\text{Meas}_{\s}^k(\M,\Om)\coloneqq\int_{\D(\M)}r^{k-n-\s}  \sup_{\mathbf{u}\in \cU([\bom]^\perp)}\chi_{\Om}(p+r\mathbf{u})\,d \H^{d}(p,\bom,r).
\end{equation}
where $\cD(\cM)\coloneqq \cD_{\rm odd}$ is defined in \eqref{Dodd}. For simplicity, assume that $\overline\M\subseteq\Om$. As is done in minimal surface theory, the first variation is computed relative to all manifolds that have the same boundary.

\begin{theorem}
Assume that $\M$ has $C^{1,\a}$ regularity with $\a>\s$. For $z\in\M$, let $\vol(z)\in \hat{\Lambda}_u^k(\R^n)$ denote a volume form for $T_z\M$, so $[\vol(z)]=T_z\M$. A necessary and sufficient condition for the vanishing of the first variation of $\text{\rm Meas}^k_\s(\M,\Om)$ with respect to compact, $k$-dimensional manifolds with the same boundary as $\M$ is that for all $z\in\M$, 

\beqn\label{variation}
\lim_{\e\downarrow 0}\left(\int_{\A_{\rm e}^+(z,\e)}-\int_{\A_{\rm o}^+(z,\e)}\right)\frac{(\aa\wedge\bom)\llcorner\vol(z)}{r^{1+\s}}\,d\H^{w+1}(\aa,\bom,r)=\bzero,
\eeqn
where $w=\text{\rm dim}\,\cW=n-1+k(n-1-k)$ and
\begin{align}
\A_{\rm o}^+(z,\e)&\coloneqq \{(\ba,\bom,r)\in\W\times(\e,\infty): \H^0(D(z+r\aa,\bom,r)\cap \M) \mbox{ is odd},\ \bom\cdot \vol(z)>0\},\\
\A_{\rm e}^+(z,\e)&\coloneqq \{(\ba,\bom,r)\in\W\times(\e,\infty): \H^0(D(z+r\aa,\bom,r)\cap \M) \mbox{ is even},\ \bom\cdot \vol(z)>0\}.
\end{align}
\end{theorem}

\begin{proof}
Fix $\eta>0$ and define
\[
f_{\eta}(p,\bom, r)\coloneqq g_{\eta}(r)\sup_{\aa\in \cU([\bom]^{\perp})}\chi_{\Om}(p+r\aa), \qquad (p,\bom,r)\in\D,
\]
where 
\[
g_{\eta}(r)\coloneqq 
\begin{cases}
0&\mbox{ if }r\le \eta,\\
r^{k-n-\s}&\mbox{ if }r>\eta,
\end{cases}
\]
so that $f_\eta$ is bounded. Furthermore, define the `truncated' $\s$-measure by 
\begin{equation}
\text{Meas}_{\s,\eta}^k(\M,\Om)\coloneqq\int_{\D(\M)}f_\eta  \,d \H^{d}.
\end{equation}
By the dominated convergence theorem, this truncated version converges to $\text{Meas}_{\s}^k(\cM,\Omega)$ as $\eta\downarrow 0$. To establish the desired result, we will compute the first variation of this truncated $\s$-measure and then send $\eta$ to zero. 

Consider $z_\circ\in\M$ and $\bw\in C^{\infty}_c(\R^n,\R^n)$ such that $\bw(z_\circ)\neq \bzero$ and the support of $\bw$, $\spt(\bw)$, does not intersect $\pa \M$. For each $t\in\R$, define the set 
\[
\M_t\coloneqq \{z+t\bw(z):z\in\M\}.
\]
Due to the regularity of $\M$, there is a sufficiently small open interval $\cI$ containing zero such that for $t\in\cI$, $\M_t$ is a $C^{1,\a}$, $k$-dimensional manifold contained in $\Om$. Let $\vol_t$ be the induced volume form on $\cM_t$ coming from the volume form $\vol$ on $\cM$. Moreover, $\overline{\cM}_t$ is a manifold with boundary such that $\partial\cM_t=\partial\cM$ for all $t\in\cI$.
\medskip

\noindent\textit{Step One}: We show that that near each $z\in\cM_t$, $\cM_t$ is contained in a set of the form \eqref{bowtie}. For each $t\in \cI$ and $z\in \M_t$, there is a parametrization of $\M_t$ near $z$ of the form 
\[
\xi_{t,z}(\bu)=z+\bu+\bphi_{t,z}(\bu),\qquad \bu\in \cE_{t,z},
\]
where $ \cE_{t,z}$ is an open neighborhood of zero in $T_z\M_t$ and $\bphi_{t,z}:\cE_{t,z}\rightarrow \R^n$ satisfies 
\beqn\label{gzprop}
\bphi_{t,z}(\bzero)=\bzero,\quad \nabla \bphi_{t,z}(\bzero)=\bzero,\quad \mbox{and}\quad \bphi_{t,z}(\bu)\in \big(T_z\M_t\big)^{\perp}, \qquad \bu\in  \cE_{t,z}.
\eeqn
As $\M_t$ is $C^{1,\a}$, so are the functions $\bphi_{t,z}$. One can choose the interval $\cI$ small enough so  that, (i) there is a $M>0$ such that for all $t\in \cI$ and $z\in \cM_t\cap \text{spt}(\bw)$, if $\bu\in T_z\cM_t$ and $|\bu|<M$, then $\bu\in \cE_{t,z}$ and (ii) the H\"{o}lder constants for $\nabla\bphi_{t,z}$ can be chosen uniformly in $t\in \cI$ and $z\in\cM_t\cap\spt(\bw)$. Thus, by \eqref{gzprop}$_{1,2}$ and using the characterization of $C^{1,\a}$ functions from Andersson \cite{A1997}, there is $K>0$ such that 
\beqn\label{uniformK}
|\bphi_{t,z}(\bu)|\le K|\bu|^{1+\a},\qquad |\bu|<M,\ z\in\M_t\cap\spt(\bw),\ t\in \cI.
\eeqn
For $z\in \M_t\cap\spt(\bw)$, let $\cR(z)$ be the set defined in \eqref{bowtie} with $\bmu=\vol_t(z)$ and $K$ as in \eqref{uniformK}. As $\bphi_{t,z}$ is a local parametrization for $\M_t$ near $z$ with \eqref{gzprop}$_3$ and \eqref{uniformK}, there is a $R>0$ such that, if $B(z,R)$ is the open ball of radius $R$ centered at $z$, then \begin{equation}\label{cMcR}
B(z,R)\cap\M_t\subseteq \overline{\cR(z)},\qquad  z\in\M_t\cap\spt(\bw),\, t\in\cI.
\end{equation}

\medskip
\noindent\textit{Step Two}: We lay the groundwork to apply Theorem~\ref{thmATTmod}, which provides a formula for the derivative with respect to $t$ of $\text{Meas}_{\s,\eta}^k(\M_t,\Om)$. Begin by setting
\[
\cA_{\pa \M}\coloneqq \pa \M\times \W \times \R^+_0\times \R^+
\quad \text{and}\quad
\cA_{\pa D}\coloneqq \M\times \W \times \R^+.
\]
If we define
\[
\Xi:\overline\M\times\W\times\R_0^+\times \R^+\rightarrow\R^n\times\Lambda_u^k(\R^n)\times\R
\]
by 
\begin{equation}\label{xithm}
\Xi(z,\textbf{a}, \bom, \xi, r)\coloneqq(z+\xi \textbf{a},\bom, r),\qquad (z,\textbf{a},\bom,\xi,r)\in \overline\M\times\W\times \R^+_0\times\R^+,
\end{equation}
then $\D_{\pa \M}\subseteq \Xi(\cA_{\pa\M})$. Similarly, for 
\[
\Psi: \overline{\M}\times \W \times \R^+\rightarrow \R^n\times \hat{\Lambda}^k_u(\R^n) \times\R^+
\]
defined by
\begin{equation}\label{Psi}
\Psi(z,\textbf{a},\bom,r)\coloneqq(z+r\textbf{a},\bom,r), \qquad (z,\textbf{a},\bom,r)\in \overline{\M}\times\W\times \R^+,
\end{equation}
we have $\D_{\pa D}\subseteq\Psi(\cA_{\pa D})$.

Let $\cN$ be the $(d-1)$-dimensional manifold defined as the disjoint union of $\cA_{\pa \M}$ and $\cA_{\pa D}$, and define the function $\Theta:\cI\times\cN\rightarrow \D$ by 
\[
\Theta(t,x)\coloneqq \begin{cases}
\Xi(x)&\mbox{ if }x\in \cA_{\pa \M},\\
\Psi(z+t\bw(z),\aa,\bom,r) &\mbox{ if }x=(z,\aa,\bom,r)\in \cA_{\pa\D}.
\end{cases}
\] 
Set $\Theta_t\coloneqq \Theta(t,\cdot)$, and write $\Theta_t'$ for the a partial derivative of $\Theta_t$ with respect to $t$. Since $\Xi$ and $\Psi$ are $C^1$, so is $\Theta$. Using the notation introduced in \eqref{cong}, it follows from Items \ref{paD2}--\ref{paMD} of Lemma \ref{lemmeasure2} and Proposition \ref{oddFP} applied to $\M_t$ that 
\[
\Theta_t(\cN)\cong\D_{\pa\M1(t)}\cup\D_{\pa D1(t)}\cong\pa^*\D(\M_t),
\]
where 
\begin{align*}
\D_{\pa \M 1}(t)&\coloneqq\{(p,\bom,r)\in\D:\H^0(\overline{D}(p,\bom,r)\cap \pa \M_t)=1 \},\\
\D_{\pa D 1}(t)&\coloneqq\{(p,\bom,r)\in \D: \H^0(\pa D(p,\bom,r)\cap \M_t)= 1\}.
\end{align*}
Also, from Items \ref{paD2}--\ref{paMD} of Lemma \ref{lemmeasure2}, we know
\[
\H^{d-1}(\{(p,\bom,r)\in\pa^*\D(\M_t):\H^0(\Theta_t^{-1}(\{(p,\bom,r)\})>1\}))=0.
\]
As the gradients of the functions $\Xi$ and $\Psi$ are injective almost everywhere on $\cA_{\pa \M}$ and $\cA_{\partial D}$, respectively, it follows that the gradient of $\Theta_t$ is also injective almost everywhere for sufficiently small $t$. Assume $\cI$ is chosen small enough so that the gradient of $\Theta_t$ is injective almost everywhere for all $t\in\cI$. Thus, $\Theta$ satisfies the conditions \ref{D1}-\ref{D3} with $\cG_t$ replaced by $\D(\M_t)$. 

To use Theorem \ref{thmATTmod} it remains to verify properties \ref{F1}-\ref{F5}. Using the notation \eqref{veldef} and \eqref{defcF} with $\cG_t$ replaced by $\D(\M_t)$, the velocity $\bv$ associated with $\Theta$ is given by 
\beqn
\bv(t,p,\bom,r)=\begin{cases}
(\bzero,\bzero,0)&\mbox{if }(p,\bom,r)\in \D_{\pa \M_1}(t),\\
(\bw(z),\bzero,0) &\mbox{if }(p,\bom,r)\in \D_{\pa D1}(t),
\end{cases}
\qquad (t,p,\bom,r)\in\pa^*\F,
\eeqn
where $z\in\M_t$ is the unique point in $\pa D(p,\bom,r)\cap \M_t$ for $(p,\bom,r)\in \D_{\pa D1}(t)$. Similar to the notation $\Theta_t$, we will use $\bv_t\coloneqq \bv(t,\cdot)$. To see that \ref{F1} is true, note that $f_\eta$ is bounded and $f_\eta\in L^1(\D(\M_t))$ since $\text{Meas}_{\s,\eta}(\M_t,\Om)$ is finite for all $t\in\cI$, as argued in Section 4 of \cite{MS25}. As $f_\eta$ is bounded, it is locally integrable. Thus, $\cH^d$-a.e.~$(p,\bom,r)$ is a Lebesgue point of $f_\eta$. Viewing $f_\eta$ as a function on $\cI\times\cD$, being independent of $t$, it follows that \ref{F2} holds.  Let $\bM_t$ be the normal from Proposition \ref{ExtNormal} with $\cM$ replaced by $\cM_t$. For any bounded interval $\cJ\subseteq \cI$, one can use Proposition \ref{ExtNormal} and the change of variables from Lemma~\ref{COVPsilem} to get 

 \begin{align*}
 \int_{\cJ}\int_{\pa^*\cD(\cM_t)}|f_\eta\bv_t\cdot\bM_t| d\H^{d-1}dt&=\int_{\cJ}\int_{\D_{\pa D1}(t)} f_\eta|\bv_t\cdot \bM_t| d\H^{d-1}dt\\
 &\leq\int_{\cJ}\int_{\M_t}\int_{\W\times(\eta,\infty)} \frac{|\bw(z)|}{r^{1+\s}}\, d\H^{m+1}(\ba,\bom,r)dzdt\,. 
 \end{align*}
  The last integral is finite since $\bw$ is smooth with compact support and $\cJ$ is bounded, so \ref{F3} holds. Using the area formula and the properties of $\Theta_t$, with $J_{\Theta_t}$ denoting the Jacobian of $\Theta_t$, notice
 \begin{align*}
 \int_{\pa^*\D(\M_t)} f_\eta\bv_t\cdot\bM_t\,\d\H^{d-1}&=\int_{\Theta_t(\cA_{\pa \D})}f_\eta\bv_t\cdot\bM_t\, d\H^{d-1}\\
 &=\int_{\cA_{\pa}}(f_\eta\circ\Theta_t)(\Theta_t'\cdot\bM_t\circ\Theta_t)J_{\Theta_t}\, d\H^{d-1}\\
 &=\int_{\M}\int_{\W}\int_\eta^\infty\frac{(\Theta'_t\cdot \bM_t\circ\Theta_t)J_{\Theta_t}}{r^{-k+n+\s}}\, drd\H^{w}dz.
 \end{align*}
From the regularity of $\Theta$, it follows that the above expression is continuous in $t$. Since $f_\eta'=0$, we know that conditions \ref{F1}-\ref{F5} hold. 

\medskip
\noindent\textit{Step Three}: We can now apply Theorem \ref{thmATTmod} to differentiate $\text{Meas}_{\s,\eta}(\M_t,\Om)$ with respect to $t$ and perform a change of variables. The theorem gives
\[
\big(\text{Meas}_{\s,\eta}(\M_t,\Om)\big)'=\int_{\pa^*\D(\M_t)}f_\eta\bv_t\cdot \bM_t\, d\H^{d-1}.
\]
Using Proposition \ref{ExtNormal}, we can plug in the expression for $\bM_t$ and apply the change of variables formula \eqref{COVPsilem}. This gives an integral over a subset of $\M_t\times \W\times (0,\infty)$, which is symmetric under the transformation $\bom\mapsto -\bom$. As the resulting integrand is even under this transformation, we can simplify the computation by only integrating over those $(\ba,\bom)\in\W$ such that $\bom\cdot \vol_t(z)>0$ and doubling the result. Utilizing Proposition \ref{wedgenorm} yields
\[
\big((\bv_t\circ\Psi)\cdot(\bM_t\circ\Psi) J_\Psi\big)(z,\ba,\bomega,r)=\frac{\bw(z)\cdot[-\ba\wedge\bomega]\llcorner \vol_t(z)r^{n-k-1}}{2^{k/2}}
\]
and, so,
\beqn\label{varMeseta}
\big(\text{Meas}_{\s,\eta}(\M_t,\Om)\big)'=\int_{\M_t}\int_\W\int_\eta^\infty\frac{2^{1-k/2}\zeta(t,z,\ba,\bom,r)\bw(z)\cdot(\ba\wedge\bom)\llcorner\vol_t(z)}{r^{1+\s}}drd\H^{w}(\ba,\bom)dz,
\eeqn
where
\beqn\label{zetader}
\zeta(t,z,\ba, \bom,r)\coloneqq \begin{cases}
-1&\mbox{if }\H^0(D(z+r\ba, \bom,r)\cap \M_t)\mbox{ is odd and } \bom\cdot \vol_t(z)>0,\\
1&\mbox{if }\H^0(D(z+r\ba, \bom,r)\cap \M_t)\mbox{ is even and } \bom\cdot \vol_t(z)>0,\\
0&\mbox{otherwise}.
\end{cases}
\eeqn
\noindent\textit{Step Four}: We now take the limit of $\big(\text{Meas}_{\s,\eta}(\M_t,\Om)\big)'$ as $\eta$ goes to zero. Define 
\begin{equation}
h_\eta(t)\coloneqq \int_{\M_t}\int_\W\int_\eta^\infty \frac{2^{1-k/2}\zeta(t,z,\ba,\bom,r)\bw(z)\cdot(\ba\wedge \bom)\llcorner\vol_t(z)}{r^{1+\s}}\, dr d\H^{w}(\ba,\bom)dz.
\end{equation}
The goal is to show that $h_\eta(t)=(\text{Meas}_{\s,\eta}(\M_t,\Om))'$, viewed as a function of $t$, converges uniformly in $t$ as $\eta\downarrow 0$ to
\beqn\label{hdef}
h(t)\coloneqq \lim_{\eta\rightarrow 0}h_\eta(t).
\eeqn
To prove that this limit exists, we start by showing that $(h_\eta(t) : \eta>0)$ is Cauchy. Notice that for $\eta>\eta'>0$, 
\begin{equation}\label{hCauchy}
|h_\eta(t)-h_{\eta'}(t)|=\left|\int_{\M_t}\int_\W\int_{\eta'}^\eta \frac{2^{1-k/2}\zeta(t,z,\ba,\bom,r)\bw(z)\cdot(\ba\wedge \bom)\llcorner\vol_t(z)}{r^{1+\s}}\, dr d\H^{w}(\ba,\bom)dz\right|.
\eeqn
For $z\in\M_t$, consider the linear transformation $\bQ_{t,z}\coloneqq \bP_{\vol_t(z)}-\bP_{\vol_t(z)}^\perp$ on $\Real^n$. One can readily check that $\bQ_{t,z}$ is symmetric and orthogonal. From the orthogonality of $\bQ_{t,z}$, it follows that 
\beqn
(\ba,\bom)\mapsto (\bQ_{t,z}\ba,\bQ_{t,z}\bom)
\eeqn
is a bijection on $\W$. Employing the notation in \eqref{cDcReta}, one can use \eqref{cMcR} to see that for sufficiently small $\eta$ it is true that 
\[
(\ba,\bom, r)\notin\D_{\cR(z)}(\eta)\Longleftrightarrow(\bQ_{t,z}\ba,\bQ_{t,z}\bom, r)\notin\D_{\cR(z)}(\eta), \qquad z\in \M_t\cap \spt(\bw).
\]
So for $(\ba,\bom, r)\notin\D_{\cR(z)}(\eta)$, it follows that 
\beqn\label{zetaeven}
\zeta(t,z,\ba,\bom, r) =\zeta(t,z,\bQ_{t,z}\ba,\bQ_{t,z}\bom,r) =1.
\eeqn
Moreover, it is true that
\beqn\label{restodd}
\bw(z)\cdot(\bQ_{t,z}\ba\wedge\bQ_{t,z}\bom)\llcorner\vol_t(z)=-\bw(z)\cdot(\ba\wedge \bom)\llcorner\vol_t(z).
\eeqn
To see that this is the case, use the definition \eqref{Lkvector}, \eqref{movingdotR}, and the symmetry of $\bQ_{t,z}$ to find that
\begin{align}
\bw(z)\cdot(\bQ_{t,z}\ba\wedge\bQ_{t,z}\bom)\llcorner\vol_t(z)&=(\bw(z)\wedge\vol_t(z))\cdot\bQ_{t,z}(\ba\wedge\bom)\\
&=\bQ_{t,z}(\bw(z)\wedge\vol_t(z))\cdot(\ba\wedge\bom)\\
&=(\bQ_{t,z}\bw(z)\wedge\bQ_{t,z}\vol_t(z))\cdot(\ba\wedge\bom)\\
&=-(\bw(z)\wedge\vol_t(z))\cdot(\ba\wedge\bom)\\
&=-\bw(z)\cdot(\ba\wedge\bom)\llcorner\vol_t(z).
\end{align}
Putting together \eqref{zetaeven} and \eqref{restodd} shows that outside of $\D_{\cR(z)}(\eta)$ the integrand of the integral on the right-hand side of \eqref{hCauchy} is odd. Thus, using Lemma \ref{0bowlem}, we conclude that
\begin{align*}
|h_\eta(t)-h_{\eta'}(t)|&=\left|\int_{\M_t}\int_{\D_{\cR(z)}(\eta)\setminus \D_{\cR(z)}(\eta')} \frac{2^{1-k/2}\zeta(t,z,\ba,\bom,r)\bw(z)\cdot(\ba\wedge\bom)\llcorner\vol_t(z)}{r^{1+\s}}\, d\H^{w+1}(\ba,\bom,r)dz\right|\\
&\le 2\int_{\M_t}\int_{\D_{\cR(z)}(\eta)}\frac{|\bw(z)|}{r^{1+\s}}\, d\H^{w+1}(\ba,\bom,r)dz\\
&\le 2\int_{\M_t} C\eta^{\a-\s}\big(-\ln[ K(2\eta)^\a]\big)^{k-1}|\bw(z)|\, dz,\\
&\le 2C\H^k(\cM_t)\eta^{\a-\s}\big(-\ln[ K(2\eta)^\a]\big)^{k-1}||\bw||_{L^\infty(\R^n)}.
\end{align*}
As $\bw$ is smooth with compact support and $\H^k(\M)$ is finite, $\cH^k(\cM_t)$ is uniformly bounded for $t\in \cI$. Thus, this proves that $(h_{\eta}(t): \eta>0)$ is uniformly Cauchy in $t$. It follows that the limit defining $h(t)$ in \eqref{hdef} exists and the limit is uniform in $t$. Therefore, we can conclude that 
\[
(\text{Meas}_{\s}(\M_t,\Om))'=\lim_{\eta\rightarrow 0}\Big((\text{Meas}_{\s,\eta}(\M_t,\Om)\Big)'=h(t).
\] 
It follows that for the variation of $\text{Meas}_\s(\M,\Om)$ in the direction $\bw$ to vanish means that 
\[
\lim_{\eta\downarrow 0}\int_{\M}\int_\W\int_\eta^\infty \frac{2^{1-k/2}\zeta(0,z,\ba,\bom,r)\bw(z)\cdot(\ba\wedge \bom)\llcorner\vol(z)}{r^{1+\s}}\, dr d\H^{w}(\ba,\bom)dz=0.
\]
As $\bw$ is arbitrary in a neighborhood of $z_\circ$, a standard argument implies that 
\[
\lim_{\e\downarrow 0}\int_\W\int_\e^\infty \frac{2^{1-k/2}\zeta(0,z_\circ,\ba,\bom,r)\bw(z)\cdot(\ba\wedge \bom)\llcorner\vol(z)}{r^{1+\s}}\, dr d\H^{w}(\ba,\bom)=0.
\]
From the definition of $\zeta$ in \eqref{zetader} and because the $z_\circ\in\M$ was arbitrary, we find that \eqref{variation} holds for all $z\in \M$.
\end{proof}

Motivated by the connection between the $k$-dimensional measure and $\text{Meas}_\sigma^k$, see \eqref{limksmeas}, and the fact that manifolds that minimize their $k$-dimensional measure relative to all smooth manifolds with the same boundary have a zero mean-curvature vector at each point, the previous theorem motivates defining a nonlocal mean-curvature vector $\bh_\s$ at $z\in\M$ by 
\beqn\label{sigcurvature}
\bh_\s(z)\coloneqq \lim_{\e\downarrow 0}\left(\int_{\A_{\rm e}^+(z,\e)}-\int_{\A_{\rm o}^+(z,\e)}\right)\frac{(\aa\wedge\bom)\llcorner\vol(z)}{r^{1+\s}}\,d\H^{w+1}(\aa,\bom,r).
\eeqn
Notice that the expression defining $\bh_\sigma(z)$ is the opposite of the expression appearing on the left-hand side of \eqref{variation}. The reason for this is that the first variation of the $k$-dimensional measure involves the opposite of the mean-curvature vector, so it seems reasonable that the first variation of the fractional measure should involve the opposite of a nonlocal mean-curvature vector. One can readily check that this vector is orthogonal to the manifold $\cM$ at $z$ by using \eqref{movingdotR}. However, in general this vector will not be parallel to the classical mean-curvature vector.

We now show that the mean-curvature vector defined in \eqref{sigcurvature} is consistent with the nonlocal mean-curvature for surfaces defined by Paroni, Podio-Guidugli, and Seguin \cite{PPGS99} when $n=k+1$. The mean curvature is obtained from the mean-curvature vector by dotting it with a unit normal. Since $n=k+1$ and we have a volume form $\vol$ for $\cM$, the Hodge star, see \eqref{Hstar}, can be used to obtain the associated unit normal. Namely, choose the unit normal $\bn=*\vol$. Moreover, for every $\bom\in\hat\Lambda^k_u(\Real^n)$ there is an associated $\bu=*\bom\in\cU(\Real^n)$. Using this notation, the fact that the Hodge star is an isometry, and the multivector identities \eqref{movingdotL}, \eqref{movingdotR}, \eqref{IPstar}, \eqref{anticom}, and \eqref{kvip} yield
\begin{align*}
\bn\cdot(\aa\wedge \bom)\llcorner \vol
&=\ba\lrcorner(\bn\wedge\vol)\cdot \bom\\
&=[(\aa\cdot \bn)\vol-\bn\wedge(\ba\lrcorner\vol)]\cdot \bom\\
&=(\aa\cdot \bn)(*\vol\cdot*\bom)-(\ba\lrcorner\vol)\cdot(\bn\lrcorner \bom)\\
&=(\ba\cdot\bn)(\bn\cdot\bu)-*(\ba\lrcorner\vol)\cdot*(\bn\lrcorner \bom)\\
&=(\ba\cdot\bn)(\bn\cdot \bu)-*\vol \wedge\ba\cdot*\bom\wedge\bn\\
&=(\ba\cdot\bn)(\bn\cdot \bu)-\bn \wedge\ba\cdot\bu\wedge\bn\\
&=\ba\cdot\bu.
\end{align*}
Also, as $\bzero=\ba\lrcorner\bom$, it follows that $\bzero=*(\ba\lrcorner \bom)=(*\bom)\wedge \ba=\bu\wedge\ba$, which implies that $\ba$ and $\bu$ are parallel. As $\vol\cdot\bom>0$ and the Hodge star is an isometry, we must have $\bn\cdot \bu>0$. It follows that $\bu=\text{sgn}(\ba\cdot\bn)\ba$, with sgn being the sign function, and, so,
\beqn
\bn\cdot(\ba\wedge\bom)\llcorner\vol=\text{sgn}(\ba\cdot \bn).
\eeqn

Thus, taking the inner-product of \eqref{sigcurvature} with $\bn(z)$ yields
\begin{align}\label{meancurv}
\bn(z)\cdot \bh_\s(z) &=\lim_{\e\downarrow 0}\left(\int_{\A_{\rm e}^+(z,\e)}-\int_{\A_{\rm o}^+(z,\e)}\right)\frac{  \text{sgn}(\ba\cdot \bn(z))  }{r^{1+\s}}\,d\H^{w+1}(\aa,\bom,r).
\end{align}
Since $\A_e^+(z,\e)$ and $\A_o^+(z,\e)$ both require $\bom\cdot \vol(z)>0$, we define $\check{\chi}_{\M}:\M\times\W\times \R^+\rightarrow \R$ by 
\beqn
\check{\chi}_\cM(z,\ba,\bom,r)\coloneqq 
\begin{cases}
1&\big( (\ba,\bom,r)\in \A_{\rm e}^+(z,\e) \text{ and }\ba\cdot \bn(z)<0\big)\\
& \quad \text{or } \big( (\ba,\bom,r)\in \A_{\rm o}^+(z,\e) \text{ and }\ba\cdot \bn(z)>0\big),\\
-1&\big( (\ba,\bom,r)\in \A_{\rm e}^+(z,\e) \text{ and }\ba\cdot \bn(z)>0\big)\\
& \quad \text{or } \big( (\ba,\bom,r)\in \A_{\rm o}^+(z,\e) \text{ and }\ba\cdot \bn(z)<0\big),\\
0& \text{otherwise. }
\end{cases}
\eeqn
It follows that \eqref{meancurv} can be written as 
\[
\bn(z)\cdot  \bh_\s(z)=\lim_{\e\downarrow 0}\int_{(\W\times\R^+)^+(z)}\frac{\check{\chi}_\M(z,\ba,\bom,r)}
{r^{1+\s}}\,d\H^{w+1}(\aa,\bom,r),
\]
where 
\[
(\W\times\R^+)^+(z)\coloneqq \{(\aa,\bom,r)\in\W\times\R^+:\bom\cdot\vol(z)>0\}.
\]
Using the change of variables $(\ba,\bom,r)\mapsto z+2r\ba$ with the area formula shows that
\beqn
\bn(z)\cdot  \bh_\s(z)=2^{1+\sigma}\omega_{n-2} H_\sigma(z),
\eeqn
which means that $\bn(z)\cdot\bh_\sigma(z)$ agrees with $H_\s(z)$ (see \eqref{Hsigma}) up to a multiplicative constant. The reason that these two quantities are not identical is because $H_\s(z)$ is normalized to ensure that it converges in the appropriate limit to the classical mean-curvature, while $\bh_\s(z)$ has not. We leave the calculation of this normalizing constant to future work.

\appendix
\section{Appendix}\label{secA}
\setcounter{section}{1} \setcounter{equation}{0} 

\subsection{A change of variables}\label{secA1}

Here we derive a useful change of variables formula involving the function
\beqn
\Psi:\overline{\cM}\times\cW\times\Real^+\rightarrow \cD
\eeqn
defined by
\beqn\label{COVPsi}
\Psi(z,\ba,\bomega,r)\coloneqq (z+r\ba,\bomega,r),\qquad (z,\ba,\bomega,r)\in\overline{\cM}\times\cW\times\Real^+.
\eeqn
See the beginning of Section~\ref{sec:3} for the introduction of the notation used here.

\begin{lemma}\label{COVPsilem}
If $\cA$ is a measurable subset of $\overline{\cM}\times\cW\times\Real^+$ and $f:\Psi(\cA)\rightarrow\Real$ is an integrable function, then
\begin{multline}\label{COVPsieqn}
\int_{\Psi(\cA)}\Big[ \sum_{(z,\ba,\bomega,r)\in\Psi^{-1}(p,\bomega,r)} f(p,\bomega,r)\Big] d\cH^{d-1}(p,\bomega,r)\\
=\int_\cA f(z+r\ba,\bomega,r) J_\Psi(z,\ba,\bomega,r) d\cH^{d-1}(z,\ba,\bomega,r),
\end{multline}
where $d=\text{\rm dim}\,\cD=n+k(n-k)+1$ and
\beqn\label{PsiJacobian}
J_\Psi(z,\ba,\bomega,r)\coloneqq\frac{r^{n-k-1}}{2^{k/2}}\sqrt{(1+r^2)|\bP_{\bomega}(\ba\lrcorner \text{\bf vol}(z))|^2+2(\text{\bf vol}(z)\cdot\bomega)^2}
\eeqn
is the Jacobian of $\Psi$.
\end{lemma}

\begin{proof}
By using a partition of unity, we need only consider the case for sets of the form $\cA=\cM_\cA\times\cW_\cA\times\cI_\cA$, where $\cM_\cA\subseteq\overline{\cM}$, $\cW_\cA\subseteq\cW$, and $\cI_\cA\subseteq\Real^+$, for which $\cM_\cA$ and $\cW_\cA$ are each covered by a single chart---meaning there is a set $M_\cA$, a subset of a halfspace of $\Real^k$, a set $W_\cA\subseteq\Real^w$, where $w= \text{dim}\,\cW=n-1+k(n-k-1)$, and diffeomorphisms $\phi:M_\cA\rightarrow\cM_\cA\subseteq\Real^n$ and $\bchi:W_\cA\rightarrow\cW_\cA\subseteq \Real^{n+{n\choose k}}$. Since $\cW_\cA\subseteq\cW$, we can represent $\bchi$ in terms of its component functions: $\bchi=(\bchi_1,\bchi_2)$, where $\bchi_1:W_\cA\rightarrow\cU(\Real^n)$ and $\bchi_2:W_\cA\rightarrow\hat\Lambda_u^k(\Real^n)\subseteq\Real^{n\choose k}$. 

Recall that if $g$ is an integrable function defined on $\cW_\cA$, then
\beqn\label{Jchi}
\int_{\cW_\cA}g(\ba,\bomega)d\cH^w(\ba,\bomega)=\int_{W_\cA}g(\bchi_1(u),\bchi_2(u))J_{\bchi} (u)du,
\eeqn
where $J_{\bchi}=\sqrt{\det(\nabla\bchi^T\nabla\bchi)}$ is the Jacobian of $\bchi$. Moreover, if $h$ is an integrable function defined on $\cM_\cA$, then
\beqn\label{Jphi}
\int_{\cM_\cA}h(z)d\cH^k(z)=\int_{M_\cA}h(\phi(x))J_\phi(x)dx.
\eeqn
Now set $A\coloneqq M_\cA\times W_\cA\times \cI_\cA$, and define $F:A\rightarrow\Psi(\cA)$ by
\beqn
F(x,u,r)\coloneqq (\phi(x)+r\bchi_1(u),\bchi_2(u),r),\qquad (x,u,r)\in A.
\eeqn
To establish \eqref{COVPsieqn} by the area formula, from \eqref{Jchi} and \eqref{Jphi} it suffices to show that
\beqn\label{JFneeded}
J_F(x,u,r)=\frac{r^{n-k-1}}{2^{k/2}}\sqrt{(1+r^2)|\bP_{\bchi_2(u)}(\bchi_1(u)\lrcorner \vol(\phi(x)))|^2+2(\vol(\phi(x))\cdot\bchi_2(u))^2} J_\bchi(u) J_\phi(x).
\eeqn
The proof of this is similar to that of Mihaila and Seguin \cite{MS25}[Lemma~5] and, so, most of the proof will be skipped. However, the last step in the calculation is different so we pick up the proof from there. Namely, using the ideas in the result just mentioned, one can arrive at
\beqn\label{JFcalc}
|\det(\nabla F^T \nabla F)|=\frac{r^{2n-2k-2}}{2^{2k-1}}\det[ \bP_\phi \big ((1+r^2)\bchi_1\otimes\bchi_1 + 2\bP_{\bchi_2})\bP_\phi^T] J_\bchi J_\phi.
\eeqn
Here the variables $x$ and $u$ are suppressed for ease of presentation, $\bP_\phi:\Real^n\rightarrow T_{\phi(x)}\overline{\cM}$ is the projection from $\Real^n$ to the tangent space of $\overline{\cM}$ at $\phi(x)$, and $\bP_{\bchi_2}:\Real^n\rightarrow\Real^n$ is the projection of $\Real^n$ onto the subspace associated with the simple $k$-vector $\bchi_2(u)$.

To compute the determinant of $\bL\coloneqq \bP_\phi((1+r^2)\bchi_1\otimes\bchi_1 + 2\bP_{\bchi_2})\bP_\phi^T$, which is a linear mapping from $T_{\phi(x)}\overline{\cM}$ to itself, we introduce an orthonormal basis of this tangent space: $\bt_1,\bt_2,\dots,\bt_k$. Assume this basis is oriented so that $\vol(\phi(x))=\bt_1\wedge\dots\wedge\bt_k$. Moreover, this basis can be chosen so that $\bt_i\cdot\bchi_1=0$ for $i=2,\dots,k$. Relative to this basis, the components of the matrix $[\bL]$ of $\bL$ are given by
\begin{align}
[\bL]_{11}&=(1+r^2)(\bt_1\cdot\ba)^2+2|\bP_{\bchi_2}\bt_1|^2,\\
[\bL]_{ij}&=2\bP_{\bchi_2}\bt_i\cdot \bP_{\bchi_2}\bt_j\qquad 2\leq i,j\leq k.
\end{align}
Letting $m_{ij}([\bL])$ denote the $ij$-minor of $[\bL]$, we can compute the determinant of $[\bL]$ by expansion along the first row:
\begin{align}
\det [\bL] &= \sum_{j=1}^k (-1)^{j+1}[\bL]_{1j}\det(m_{1j}([\bL]))\\
&=[(1+r^2)(\bt_1\cdot\bchi_1)^2+2|\bP_{\bchi_2}\bt_1|^2]\det(m_{11}([\bL]))+\sum_{j=2}^k (-1)^{j+1}[\bL]_{1j}\det(m_{1j}([\bL]))\\
&=(1+r^2)(\bt_1\cdot\bchi_1)^2\det(m_{11}([\bL]))+\sum_{j=1}^k (-1)^{j+1}2\bP_{\bchi_2}\bt_1\cdot \bP_{\bchi_2}\bt_j\det(m_{1j}([\bL])).\label{detL}
\end{align}
We will complete the calculation by considering the two terms in the last equation separately. For the first one, notice that since $\bt_i\cdot\bchi_1=0$ for $i\not=1$, by \eqref{Lkvector} and \eqref{kvip},
\begin{align}
(\bt_1\cdot\bchi_1)^2\det(m_{11}([\bL]))&=2^{k-1}(\bt_1\cdot\bchi_1)^2(\bP_{\bchi_2}\bt_2\wedge\dots\wedge\bP_{\bchi_2}\bt_k)\cdot(\bP_{\bchi_2}\bt_2\wedge\dots\wedge\bP_{\bchi_2}\bt_k)\\
&=2^{k-1}|(\bt_1\cdot\bchi_1)\bP_{\bchi_2}(\bt_2\wedge\dots\wedge\bt_k)|^2\\
&=2^{k-1}|\bP_{\bchi_2}(\bchi_1\lrcorner \vol)|^2.
\end{align}
Next, notice that the second term on the right-hand side of \eqref{detL} is the determinant expansion along the first row of $2\bP_\phi\bP_{\bchi_2}\bP_\phi^T$. Thus, from the computation in Mihaila and Seguin \cite{MS25}\footnote{See the equation before (6.28) in \cite{MS25}.} we have
\beqn
\sum_{j=1}^k (-1)^{j+1}2(\bP_{\bchi_2}\bt_1\cdot \bP_{\bchi_2}\bt_j)\det(m_{1j}([\bL]))=\det(2\bP_\phi\bP_{\chi_2}\bP_\phi^T)=2^k(\vol\cdot\bchi_2)^2.
\eeqn
Combining these last two calculations with \eqref{JFcalc} and \eqref{detL} yields \eqref{JFneeded}, as desired.
\end{proof}

\subsection{A transport theorem}\label{secA2}

To compute the first variation of the fractional $k$-dimensional measure, we will make use of the below transport theorem, whose proof was given by Seguin \cite{S20c}.

\begin{theorem}\label{thmATTmod}
Let $\cP$ be a $d$-dimensional submanifold of $\Real^N$. For each $t\in\cI$, $\cI$ being an open interval of $\Real$, let $\cG_t$ be an open subset of $\cP$ that is locally of finite perimeter with exterior unit normal $\bnu(t,\cdot)$ such that there exists a $(d-1)$-dimensional Riemannian manifold $\cN$ and a function $\Theta\in C^1(\cI\times\cN,\Real^N)$ such that for all $t\in\cI$ the following conditions hold:
\begin{enumerate}[label=(\textit{D\arabic*})]
\item \label{D1} the gradient of $\Theta_t\coloneqq \Theta(t,\cdot):\cN\rightarrow\Real^N$ is injective $\cH^{d-1}$-a.e.,
\item \label{D2} $\partial^*\cG_t$ and $\Theta_t(\cN)$ differ by a set of $\cH^{d-1}$-measure zero, and
\item \label{D3} $\cH^{d-1}(\{ x\in\partial^*\cG_t\ |\ \cH^0(\Theta_t^{-1}(\{x\}))>1\})=0$.
\end{enumerate}
The `velocity' $\bv$ associated with $\Theta$ is defined by
\beqn\label{veldef}
\bv(t,x)\coloneqq \Theta'(t,\Theta_t^{-1}(x))\qquad \text{for $\cH^{d-1}$-a.e.}\ x\in \partial^*\cG_t,\, t\in\cI,
\eeqn
where the prime denotes the partial derivative with respect to $t$. Set
\beqn\label{defcF}
\cF\coloneqq \{(t,x)\in\cI\times\cP\ |\ x\in \cG_t\}\subseteq \cI\times\cP,
\eeqn
and consider $f:\cI\times\cP\rightarrow\Real$ such that
\begin{enumerate}[label=(\textit{F\arabic*})]
\item \label{F1} $f\in L^\infty(\cF)$ and $f(t,\cdot)\in L^1(\cG_t)$ for all $t\in\cI$,
\item \label{F2} $\cH^d$-a.e.~$(x,t)\in\cI\times\cP$ is a Lebesgue point of $f$,
\item \label{F3} for any bounded interval $\cJ\subseteq\cI$, $\int_\cJ\int_{\partial^*\cG_t}|f(t,\cdot)\bv(t,\cdot)\cdot\bnu(t,\cdot)|\, d\cH^{d-1}dt<\infty$,
\item \label{F4} the function $t\mapsto \int_{\partial^*\cG_t}f(t,\cdot)\bv(t,\cdot)\cdot\bnu(t,\cdot)\, d\cH^{d-1}$ is continuous on $\cI$, and
\item \label{F5} $f'=0$.
\end{enumerate}
It follows that for $t\in\cI$,
\beqn\label{eqATT}
\Big(\int_{\cG_t} f(t,\cdot)\, d\cH^d \Big)'= \int_{\partial^*\cG_t} f(t,\cdot) \bv(t,\cdot)\cdot\bnu(t,\cdot)\, d\cH^{d-1}.
\eeqn
\end{theorem}

\bibliography{nonmeasure}
\bibliographystyle{is-alpha}

\end{document}